\documentclass[12pt]{amsart}
\usepackage{amsthm,amsmath,amssymb,amscd}

{\end{list}}
{
   \newtheorem{theorem}{Theorem}[section]
   \newtheorem{proposition}[theorem]{Proposition}
   \newtheorem{lemma}[theorem]{Lemma}

   \newtheorem{corollary}[theorem]{Corollary}

}
{\theoremstyle{definition}
   \newtheorem{definition}[theorem]{Definition}
   \newtheorem{remark}[theorem]{Remark}

}

\newcommand{\RR}{{\mathbb{R}}}

\newcommand{\DD}{{\mathbb{D}}}
\newcommand{\QQ}{{\mathbb{Q}}}

\newcommand{\ZZ}{{\mathbb{Z}}}

\newcommand{\cA}{{\mathcal A}}

\newcommand{\cF}{{\mathcal F}}
\newcommand{\cG}{{\mathcal G}}

\newcommand{\cL}{{\mathcal L}}

\newcommand{\fM}{{\mathfrak{M}}}

\newcommand{\zz}{{\underline{o}}}
\newcommand{\Gr}{\operatorname{Gr}}
\newcommand{\img}{\operatorname{img}}

\newcommand{\res}{\operatorname{res}}

\newcommand{\Id}{\operatorname{Id}}

\newcommand{\Span}{\operatorname{Span}}

\newcommand{\Hom}{{\operatorname{Hom}}}

\newcommand{\Homd}{\operatorname{Hom}^\bullet}
\newcommand{\cHomd}{\operatorname{\mathcal Hom}^\bullet}

\newcommand{\Ext}{{\operatorname{Ext}}}

\newcommand{\codim}{\operatorname{codim}}

\newcommand{\isom}{\simeq}

\newcommand{\LO}{Lefschetz operation }
\newcommand{\Po}{Poincare }

\newcommand{\HL}{Hard Lefschetz }
\newcommand{\st}{\operatorname{star}}

\newcommand{\prim}{\operatorname{Prim}}
\newcommand{\Cd}{C^{\bullet}}
\newcommand{\Int}{\operatorname{Int}}

\newcommand{\tensor}{\otimes}

\newcommand{\setmin}{\,\protect%
\begin{picture}(8,10)\qbezier(1,5.5)(4,4.)(7,2.5)\end{picture}\,}

\begin{document}
\title{Relative Hard Lefschetz Theorem for Fans.}

%\author{Balin Fleming}
\author{Kalle Karu}
\thanks{This work was partially supported by NSERC Discovery grant.}
\address{Mathematics Department\\ University of British Columbia \\
  1984 Mathematics Road\\
Vancouver, B.C. Canada V6T 1Z2}
\email{karu@math.ubc.ca}
% \date{Dec 1, 2001}

\begin{abstract}
We prove the Relative Hard Lefschetz theorem and the Relative Hodge-Riemann bilinear relations for combinatorial intersection cohomology sheaves on fans.
\end{abstract}

\maketitle

%\addtocounter{section}{-1}

\tableofcontents

\section{Introduction.}

We consider fans as in the theory of toric varieties \cite{Oda,Fulton} and the combinatorial intersection cohomology sheaves on fans \cite{BBFK, BL1}. Our main goal is to prove the Relative Hard Lefschetz (RHL) theorem and the Relative Hodge-Riemann (RHR) bilinear relations for such sheaves. 
To a rational fan one can associate a toric variety. Then the combinatorial RHL and RHR theorems can be deduced from the corresponding theorems in geometry (see for example \cite{dCMM}). When the fan is not rational then there is no associated toric variety and the theorems give new information about the combinatorics of fans.

Let us introduce more notation and describe the main results. Recall from \cite{BBFK, BL1} that we give a fan $\Sigma$ the structure of a ringed space (over $\RR$) and consider the category $\fM = \fM(\Sigma)$ of graded locally free flabby sheaves of finite type on $\Sigma$. For every cone $\sigma\in\Sigma$ there is an indecomposable sheaf $\cL^\sigma$ in $\fM$. Moreover, every sheaf $\cF$ in $\fM$ is a direct sum of shifted copies of the sheaves $\cL^\sigma$. We write this decomposition as 
\[ \cF = \bigoplus_{\sigma\in\Sigma} W_\sigma\tensor \cL^\sigma, \]
for some graded vector spaces $W_\sigma$. 

The sheaf $\cL_\Sigma = \cL^\zz$, where $\zz$ is the minimal cone in $\Sigma$, is called the intersection cohomology sheaf of $\Sigma$. The global sections of this sheaf form the equivariant intersection cohomology of $\Sigma$, $IH_T(\Sigma)$. Let $V$ be the vector space containing $\Sigma$ and $A$ the ring of polynomial functions on $V$. By convention, linear functions in $A$ have degree $2$. For a complete fan $\Sigma$, $IH_T(\Sigma)$ is a free $A$-module. The quotient of $IH_T(\Sigma)$ by the maximal homogeneous ideal $m \subset A$ gives the intersection cohomology $IH(\Sigma)$.

We are mainly interested in the case where $\pi:\hat\Sigma\to\Sigma$ is a subdivision of fans and $\cF = \pi_* \cL_{\hat\Sigma}$. The sheaf $\cF$ lies in $\fM(\Sigma)$ and has a decomposition as above. The dimension vectors of the graded spaces $W_\sigma$ can be found combinatorially from the posets of the fans and the map $\pi$ on the posets. These vectors were called local $h$-vectors by Stanley \cite{Stanley}. The combinatorics of these vectors has been studied extensively, see for example \cite{KatzStapledon}. The main results of the relative theorems state that the spaces $W_\sigma$ behave like the intersection cohomology spaces of projective varieties: they satisfy \Po duality, Hard Lefschetz theorem and Hodge-Riemann bilinear relations.

Let $\pi: \hat\Sigma \to \Sigma$ be a subdivision of fans, and let $\cF= \pi_* \cL_{\hat\Sigma}$ be decomposed as above. A piecewise linear function $\hat{l}$ on $\hat\Sigma$ defines a degree $2$ map $\hat{l}: W_\sigma\to W_\sigma$ for every $\sigma$. We say that $\hat{l}$ is relatively strictly convex with respect to $\pi$ if it is strictly convex on the inverse image of every $\sigma\in\Sigma$.

\begin{theorem}[Relative Hard Lefschetz] \label{thm-RHL}
Let $\pi: \hat\Sigma\to\Sigma$ be a subdivision, $\hat{l}$ relatively strictly convex with respect to $\pi$, and $\pi_*(\cL_{\hat\Sigma}) = \oplus_\sigma W_\sigma \otimes \cL^\sigma$. Then $\hat{l}$ defines a Lefschetz operation on $W_\sigma$ for every $\sigma\in\Sigma$, i.e.
\[ \hat{l}^i: W_\sigma^{\dim\sigma-i} \to W_\sigma^{\dim\sigma+i}\]
is an isomorphism for every $i>0$.
\end{theorem}

The RHL theorem holds more generally for projective morphisms of fans. However, the case of projective morphisms can be reduced to the case of subdivisions as follows. Recall that a morphism of fans $\pi:\hat\Sigma\to\Sigma$ is defined by a linear map of vector spaces $p: \hat{V}\to V$, where $\hat\Sigma$ lies in $\hat{V}$ and $\Sigma$ lies in $V$. The morphism is proper if $p^{-1} (|\Sigma|) = |\hat\Sigma|$, where $|\Sigma|$ denotes the support of $\Sigma$. The morphism is projective if, in addition, there exists a piecewise linear function $\hat{l}$ on $\hat\Sigma$ that is relatively strictly convex with respect to $\pi$. Given a proper morphism $\pi:\hat\Sigma\to\Sigma$, we construct a new fan $\Sigma'$ in $\hat{V}$,
\[ \Sigma' = \{ p^{-1}(\sigma) | \sigma\in\Sigma\}.\]
This is not a fan as defined in \cite{Oda, Fulton} because its cones are not pointed (that means, not strictly convex). However, the theory of sheaves on it works the same way as for ordinary fans. Now $\pi: \hat\Sigma\to\Sigma$ factors through $\Sigma'$ and the map $\hat\Sigma\to\Sigma'$ is a subdivision to which Theorem~\ref{thm-RHL} applies. 

All theorems in this section about subdivisions $\pi: \hat\Sigma\to\Sigma$ are also valid when the fan $\Sigma$ has non-pointed cones, but the fan $\hat\Sigma$ is an ordinary fan with pointed cones. If $\hat\Sigma$ has non-pointed cones, then $\dim \sigma$ has to be replaced with $\dim\sigma-\dim\hat\zz$, where $\hat\zz$ is the smallest cone in $\hat\Sigma$. The Lefshetz isomorphism is then 
\[ \hat{l}^i: W_\sigma^{\dim\sigma-\dim\hat\zz-i} \to W_\sigma^{\dim\sigma-\dim\hat\zz+i}.\]
This version of the RHL theorem is stated in Theorem~\ref{thm-RHL1} below. 

The RHL theorem is proved together with the Relative Hodge-Riemann (RHR) bilinear relations. To define these relations we need the \Po duality pairing in the intersection cohomology. This pairing induces a nondegenerate symmetric bilinear pairing $\langle\cdot,\cdot\rangle$ on $W_\sigma$ for $\sigma\in\Sigma$:
 \[ W_\sigma^{\dim\sigma-i} \times W_\sigma^{\dim\sigma+i} \to \RR.\]
 Using the pairing and the map $\hat{l}$, we define the quadratic form $Q_{\hat{l}}$ on $W_\sigma^{\dim\sigma - i}$:
 \[ Q_{\hat{l}} (w) = \langle \hat{l}^i w, w\rangle.\]
 
 \begin{theorem}[Relative Hodge-Riemann bilinear relations] \label{thm-RHR}
 Let $\pi: \hat\Sigma\to\Sigma$ be a subdivision, $\hat{l}$ relatively strictly convex with respect to $\pi$, and $\pi_*(\cL_{\hat\Sigma}) = \oplus_\sigma W_\sigma \otimes \cL^\sigma$. Then the quadratic form 
 \[ (-1)^{\frac{\dim\sigma-i}{2}} Q_{\hat{l}}\]
 is positive definite on the primitive part of $W_\sigma^{\dim\sigma - i}$:
 \[ \prim_{\hat{l}} W_\sigma^{\dim\sigma - i} = \ker \big(\hat{l}^{i+1}: W_\sigma^{\dim\sigma-i} \to W_\sigma^{\dim\sigma+i+2}\big)\] 
 for every $i\geq 0$ and $\sigma\in\Sigma$.
 \end{theorem}
 
 Note that $W_\sigma$ in the theorem is nonzero only in even degrees, hence $\frac{\dim\sigma-i}{2}$ is an integer in the nontrivial cases. 
 
 %The theorem as stated holds when $\Sigma$ is a general fan, with possibly non-pointed cones, but $\hat\Sigma$ has pointed cones. If $\hat\Sigma$ also has non-pointed cones, we need to replace $\dim \sigma$ with $\dim\sigma-\dim\hat\zz$ in the statement.

We also give two local versions of the RHL and RHR theorems. Theorems~\ref{thm-RHL}-\ref{thm-RHR} can be reduced to these local versions. 

Consider an n-dimensional fan $\Phi$ with boundary $\partial\Phi$. We say that $\Phi$ is convex if the support $|\Phi|$ is convex. Convex fans are special cases of quasi-convex fans studied in \cite{BBFK}. On a convex (or quasi-convex) fan $\Phi$ we consider global sections of $\cL_\Phi$ that vanish on $\partial\Phi$,
\[ \cL_\Phi(\Phi,\partial\Phi) \subset  \cL_\Phi(\Phi).\]
The space of sections and the space of sections vanishing on the boundary are both free $A$-modules. Reducing modulo the maximal homogeneous ideal $m \subset A$ gives a linear map of graded vector spaces 
\[ IH(\Phi,\partial\Phi)\to IH(\Phi).\]
Define $W_\Phi$ to be the image of this map. A piecewise linear function $\hat{l}$ on $\Phi$ defines a degree $2$ map $\hat{l}: W_\Phi \to W_\Phi$. The \Po duality pairing
\[ IH^{n-i}(\Phi,\partial\Phi) \times IH^{n+i}(\Phi) \to \RR \]
restricts to a nondegenerate symmetric bilinear pairing $\langle \cdot,\cdot\rangle$ on $W_\Phi$. Define the quadratic form $Q_{\hat{l}}$ on $W_\Phi^{n - i}$ by
\[ Q_{\hat{l}} (w) = \langle \hat{l}^i w, w\rangle.\]

\begin{theorem} \label{thm-convex}
 Let $\Phi$ be a convex fan of dimension $n$ and $\hat{l}$ a strictly convex piecewise linear function on $\Phi$. Let $W_\Phi = \img \big(IH(\Phi,\partial\Phi)\to IH(\Phi)\big)$.
 \begin{enumerate}
 \item The map $\hat{l}$ defines a Lefschetz operation on $W_\Phi$, i.e.
 \[ \hat{l}^i: W_\Phi^{n-i} \to W_\Phi^{n+i}\]
is an isomorphism for every $i>0$.
 \item The quadratic form
 \[ (-1)^{\frac{n-i}{2}} Q_{\hat{l}}\]
 is positive definite on the primitive part of $W_\Phi^{n - i}$:
 \[ \prim_{\hat{l}} W_\Phi^{n - i} = \ker \big(\hat{l}^{i+1}: W_\Phi^{n-i} \to W_\Phi^{n+i+2}\big)\]  
for every $i\geq 0$.
\end{enumerate}
 \end{theorem}

If $\Phi$ has non-pointed cones, then $n$ has to be replaced with $n-\dim \zz_\Phi$ in the statement of the theorem.

We consider the support of a convex fan $\Phi$ as a cone $\sigma$ (not necessarily pointed). Then $\Phi$ is a subdivision of $\sigma$. Moreover, $W_\sigma = W_\Phi$ with the  same action of $\hat{l}$ and the same quadratic form $Q_{\hat{l}}$. This shows the equivalence of Theorem~\ref{thm-convex} to the RHL and RHR theorems.

%Hard Lefschetz theorem for quasi-convex fans has been conjectured by several authors. Theorem~\ref{thm-convex}(1) is the version that can be proved using convexity and Hodge-Riemann bilinear relations. One can ask about the combinatorial significance of the theorem: is the dimension vector of $W_\Phi$ determined by the poset of $\Phi$? As was explained above, we consider $\Phi$ as a subdivision of the cone $\sigma$. It follows that the dimension vector of $W_\Phi$ is determined by the poset of $\Phi$, the face poset of $\sigma$, and the map between the posets.

For the final version of the RHL and RHR theorems we consider a complete fan $\hat\Sigma$ of dimension $n$. Assume that we have piecewise linear functions $l$ and $\hat{l}$ on $\hat\Sigma$, where $\hat{l}$ is strictly convex and $l$ is only convex. We consider $l\cdot IH(\hat\Sigma)$, the image of $l$ acting on $IH(\hat\Sigma)$. Since the actions of $l$ and $\hat{l}$ commute, $\hat{l}$ defines a degree $2$ map on $l\cdot IH(\hat\Sigma)$. We use the \Po pairing on $IH(\hat\Sigma)$ to define the quadratic form $Q_{\hat{l}}$ on $l\cdot IH^{n-i}(\hat\Sigma)$:
\[ Q_{\hat{l}}(lh) = \langle \hat{l}^{i-1} l h, h \rangle.\]

\begin{theorem} \label{thm-complete}
 Let $\hat\Sigma$ be a complete fan of dimension $n$, $\hat{l}$ a strictly convex and $l$ a convex piecewise linear function on $\hat\Sigma$. Let $W = l\cdot IH(\hat\Sigma)$.
 \begin{enumerate}
 \item The map $\hat{l}$ defines a Lefschetz operation on $W$, i.e.
 \[ \hat{l}^{i-1}: W^{n-i+2} \to W^{n+i}\]
is an isomorphism for every $i>1$.
 \item The quadratic form
 \[ (-1)^{\frac{n-i}{2}} Q_{\hat{l}}\]
 is positive definite on the primitive part of $W^{n - i+2}$:
 \[ \prim_{\hat{l}} W^{n - i+2} = \ker \big(\hat{l}^{i}: W^{n-i+2} \to W^{n+i+2}\big)\]
  for every $i\geq 1$.
\end{enumerate}
 \end{theorem}

When $l=\hat{l}$ is strictly convex then the theorem follows from the usual Hard Lefschetz theorem and Hodge-Riemann bilinear relations. If $\hat\Sigma$ has non-pointed cones, then again $n$ needs to be replaced with $n-\dim \hat\zz$ in the statement of the theorem.

We reduce Theorem~\ref{thm-convex} to Theorem~\ref{thm-complete}, with $\hat\Sigma$ being the fan $\partial\Phi$, projected from a ray in the interior of $\Phi$. In particular, $\hat\Sigma$ has smaller dimension than $\Phi$. Theorem~\ref{thm-complete} in turn is proved using the RHL and RHR theorems. We define $\Sigma$ to be the fan on which $l$ is strictly convex. Then $\hat\Sigma$ is a subdivision of $\Sigma$ and the space $W = l\cdot IH(\hat\Sigma)$ can be studied using the RHL and RHR theorems applied to this subdivision.

In the following sections we will generalize the RHL and RHR theorems. First, we will allow all fans to have non-pointed cones. Second, we consider not only the sheaf $\cF=\pi_* \cL_{\hat\Sigma}$, but also sheaves $\cF=\pi_* \cL^\tau$ for $\tau\in\hat\Sigma$.
It turns out that to get the cleanest forms of the theorems, we need to proceed as in the study of perverse sheaves on algebraic varieties. We define the perverse filtration $\tau_{\leq p} \cF$ on sheaves $\cF$ in $\fM$, and the categories $\fM^{\leq p}$ and $\fM^{\geq p}$. Then the category $\fM^0 = \fM^{\leq 0} \cap \fM^{\geq 0}$ is the category of perverse sheaves, which are finite direct sums of sheaves $\cL^\tau(\codim\tau)$ for $\tau\in\Sigma$. The shift in grading by $\codim \tau$ is analogous to what happens in geometry: on a smooth $n$-dimensional variety $X$ the sheaf $\QQ_X[n]$ is perverse. We prove the RHL theorem for all sheaves $\cF = \pi_* \cG$, where $\cG$ is in $\fM^0(\hat\Sigma)$.

The outline of the paper is as follows. In Section~2 we recall the definition of the category $\fM$, the decomposition theorem and the Hard Lefschetz theorem. In Section~3 we consider the intersection pairing from \cite{BL1,BL2}. We will adapt it to the case of fans with non-pointed cones and recall the Hodge-Riemann bilinear relations. In Section~4 we define the filtration $\tau_{\leq p}$ and state the RHL and RHR theorems. These theorems are proved in Section~5.

{\bf Acknowledgment.} I thank Balin Fleming for fruitful discussion during the preparation of this article.

\section{The category $\fM$.}

In this section we recall the definitions and basic properties of locally free flabby sheaves on a fan $\Sigma$ \cite{BBFK, BL1}. We will mainly use the notation from \cite{BL1}.

\subsection{Sheaves on fans.}
We work over the field $\RR$ and consider fans $\Sigma$ in a vector space $V\isom \RR^n$ as in the theory of toric varieties \cite{Oda, Fulton}. We do not require fans to be rational. 
In \cite{Oda, Fulton} cones in a fan are pointed (the largest subspace they contain is the zero space). We consider more general fans of the form 
\[ \Sigma \times\RR^m = \{ \sigma\times\RR^m| \sigma\in\Sigma\},\]
 where $\Sigma$ is a fan with pointed cones. We call a fan $\Sigma$ pointed if all its cones are pointed.

The {\em support} $|\Sigma|$ of a fan is the union of its cones. A fan is {\em complete} if its support is all of $V$, it is {\em convex} if its support is convex and full-dimensional in $V$. A full-dimensional fan with boundary $(\Sigma, \partial\Sigma)$ is {\em quasi-convex} (see \cite{BBFK}) if $\partial \Sigma$ is a fan over a generalized homology sphere. Complete fans are special cases of quasi-convex fans with $\partial\Sigma=\emptyset$. 

A continuous function $|\Sigma|\to \RR$ is called piecewise polynomial if it is polynomial on every cone of $\Sigma$. On a convex fan it makes sens to talk about convex piecewise linear functions. Such a function is strictly convex if the largest pieces where it is linear are the maximal cones of $\Sigma$.

If $\Sigma$ is a fan in $V$ and $\hat\Sigma$ a fan in $\hat{V}$, then a map of fans $\pi: \hat\Sigma\to\Sigma$ is defined by a linear map $p: \hat{V}\to V$ that takes each cone in $\hat\Sigma$ to some cone in $\Sigma$. Then  $\pi(\sigma)$ is defined as the smallest cone in $\Sigma$ containing $p(\sigma)$. The map $\pi$ is proper if $p^{-1}(|\Sigma|) = |\hat\Sigma|$. The map is a subdivision if it is proper, $\hat{V} = V$, and $p=\Id_V$.

The set of cones in a fan forms a poset, $\tau \leq \sigma$ means that $\tau$ is a face of $\Sigma$. We let $\zz = \zz_\Sigma \in \Sigma$ be the smallest cone of $\Sigma$. (For our generalized fans $\zz=\RR^m$.) Denote
\begin{align*}
[\sigma] &= \{\tau \leq \sigma\}, \\
\partial \sigma &= [\sigma]\setmin \{\sigma\}, \\
\st\sigma = \st_\Sigma \sigma &= \{\tau \in \Sigma| \tau\geq \sigma\}\\
[\st\sigma] = [\st_\Sigma \sigma] &= \{\tau \in \Sigma| \tau\leq \nu \geq \sigma \text{ for some $\nu$}\}.
\end{align*}
We consider a cone $\sigma$ always in a vector space $V$ and define $\codim\sigma = \dim V -\dim\sigma$.

In all notations we omit the subscript $\Sigma$ if the fan is clear from the context.

A fan $\Sigma$ is given the topology in which open sets are the subfans of $\Sigma$. Then, for example, $[\sigma]\subset \Sigma$ is open, it is the smallest open set containing $\sigma$; $\st \sigma \subset \Sigma$ is the closure of $\{\sigma\}$.

The usual notion of a {\em sheaf} $F$ of vector spaces on the topological space $\Sigma$ is equivalent to the following data
\begin{itemize}
  \item a vector space $F_\sigma$ for each $\sigma\in\Sigma$,
  \item a linear map $\res^\sigma_\tau: F_\sigma\to F_\tau$ for each $\tau\leq\sigma$,
\end{itemize}
such that $\res^\sigma_\sigma=  \Id_{F_\sigma}$ and $\res^\tau_\rho \circ \res^\sigma_\tau = \res^\sigma_\rho$ for $\rho\leq\tau\leq\sigma$. The vector spaces $F_\sigma$ are the stalks of the sheaf and the maps $\res^\sigma_\tau$ are the restriction maps.

If $S\subset \Sigma$ is a subfan and $F$ a sheaf on $\Sigma$, we write $F(S)$ for the space of sections,
\[ F(S)  = \{(f_\sigma)_{\sigma \in S} \in \oplus_\sigma F_\sigma | \res^\sigma_\tau (f_\sigma) = f_\tau \text{ for $\tau\leq\sigma \in S$}\}.\]
In particular, $F([\sigma]) = F_\sigma$, and $F(\Sigma)$ is the space of global sections. We also denote the space of global sections by $\Gamma(\Sigma,F)$.

Given a subdivisions of fans $\pi:\hat{\Sigma}\to\Sigma$ and a sheaf $F$ on $\hat\Sigma$, the pushforward sheaf $\pi_* F$ has stalks
\[ \pi_*(F)_\sigma =  F(\pi^{-1}([\sigma])) = 
\{(f_{\sigma'})_{\pi(\sigma') \in[\sigma]} | 
\res^{\sigma'}_{\tau'} (f_{\sigma'}) = 
f_{\tau'} \text{ for $\tau'\leq\sigma' \in \pi^{-1}([\sigma])$}\}.\]

A {\em morphism} of sheaves $\phi:F\to G$ on $\Sigma$ is a collection of linear maps $\phi_\sigma: F_\sigma\to G_\sigma$ for $\sigma\in\Sigma$, commuting with the restriction maps.

\subsection{Definition of the category $\fM$ and the decomposition theorem.}

The {\em structure sheaf} $\cA = \cA_\Sigma$ has stalk $\cA_\sigma$, the ring of polynomial functions on $\sigma$, and restriction maps the restrictions of functions. The sheaf $\cA_\Sigma$ is a sheaf of graded $\RR$-algebras. Its global sections are piecewise polynomial functions on $\Sigma$. We follow the convention that piecewise linear functions have degree $2$. Let $A$ be the ring of polynomial functions on the ambient space $V$. Then $\cA_\Sigma$ is also a sheaf of graded $A$-algebras.

There is an obvious notion of sheaves of $\cA_\Sigma$ modules on $\Sigma$. For such sheaves $\cF$ the stalks $\cF_\sigma$ have to be $\cA_\sigma$-modules and the restriction maps $\res^\sigma_\tau$ morphisms of $\cA_\sigma$-modules. A sheaf homomorphism $\phi: \cF\to \cG$ is a homomorphism of $\cA_\Sigma$-modules if $\phi_\sigma: \cF_\sigma\to \cG_\sigma$ is a morphism of $\cA_\sigma$-modules for each $\sigma\in\Sigma$. Note that morphisms of graded sheaves must preserve the grading.

Let $\fM = \fM(\Sigma)$ denote the full subcategory of sheaves of $\cA_\Sigma$ modules that are graded, locally free, flabby, and of finite type. This means that a sheaf of $\cA_\Sigma$-modules $\cF$ lies in $\fM$ if for every $\sigma\in\Sigma$
\begin{itemize}
  \item the stalk $\cF_\sigma$ is a graded free $\cA_\sigma$ module of finite rank,
  \item the restriction map
  \[ \cF_\sigma = \cF([\sigma]) \to \cF(\partial\sigma) \]
  is surjective.
\end{itemize}

Sheaves in $\fM$ have a very simple structure. For every $\tau\in\Sigma$ there exists an indecomposable sheaf $\cL^\tau = \cL^\tau_\Sigma$ in $\fM$. The sheaf $\cL^\tau$ is supported on $\st \tau$ and it is constructed inductively as follows. Set $\cL^\tau_\tau = \cA_\tau$, and for $\sigma>\tau$ let $\cL^\tau_\sigma$ be the graded free $\cA_\sigma$-module of smallest rank that surjects onto $\cL^\tau(\partial\sigma)$. (In other words, $\cL^\tau_\sigma\to \cL^\tau(\partial\sigma)$ is a minimal projective cover of the $\cA_\sigma$-module $\cL^\tau(\partial\sigma)$.) The sheaf $\cL^\zz$ is denoted $\cL_\Sigma$ and called the intersection cohomology sheaf of $\Sigma$.

\begin{theorem}[ Decomposition Theorem]
Every sheaf $\cF$ in $\fM$ can be decomposed as a finite direct sum of shifts of the sheaves $\cL^\tau$.  
\end{theorem}

We denote by $\cF(j)$ the shift in grading, $\cF(j)^i = \cF^{j+i}$, and write a decomposition of $\cF$ as
\[ \cF = \bigoplus_{i} \cL^{\tau_i}(j_i) = \bigoplus_{\tau\in\Sigma} W_\tau \otimes \cL^\tau,\]
where $W_\tau$ are some graded vector spaces. The decomposition is not unique. However, the dimensions of the graded pieces of $W_\tau$ are determined by $\cF$. 

When $\pi:\hat\Sigma\to\Sigma$ is a subdivision of fans, then the push-forward $\pi_*$ maps $\fM(\hat\Sigma)$ to $\fM(\Sigma)$. 

\subsection{Intersection cohomology.}

Let $\Sigma$ be a fan in $V\isom \RR^n$ and recall that $A \isom \RR[x_1,\ldots,x_n]$ is the ring of polynomial functions on $V$. A sheaf $\cF$ in $\fM=\fM(\Sigma)$ is a sheaf of $A$-modules and its set of global sections $\cF(\Sigma)$ is an $A$-module. We let 
\[ \overline{\cF(\Sigma)} = \cF(\Sigma)/m \cF(\Sigma),\]
where $m$ is the maximal homogeneous ideal in $A$, $m\isom (x_1,\ldots,x_n)$. We denote the intersection cohomology
\begin{align*}
  IH(\Sigma) &= \overline{\cL_\Sigma(\Sigma)},\\
  IH(\st \tau) &= \overline{\cL^\tau(\Sigma)}.
\end{align*}
These are graded vector spaces.

When the fan $\Sigma$ is complete, or more generally quasi-convex, then $\cF(\Sigma)$ is a free $A$-module for any $\cF$ in  $\fM$. For a quasi-convex fan $(\Sigma,\partial\Sigma)$ let $\cF(\Sigma,\partial\Sigma)$ be the set of global sections that vanish on $\partial\Sigma$,
\[\cF(\Sigma,\partial\Sigma) = \ker \big( \cF(\Sigma) \to \cF(\partial\Sigma)\big).\]
 This is again a free $A$-module. We denote
\begin{gather*}
IH(\Sigma,\partial\Sigma) = \overline{\cL_\Sigma(\Sigma,\partial\Sigma)}= \cL_\Sigma(\Sigma,\partial\Sigma)/ m \cL_\Sigma(\Sigma,\partial\Sigma),\\
IH(\st\tau,\partial\st\tau) = \overline{\cL^\tau(\Sigma,\partial\Sigma)}= \cL^\tau(\Sigma,\partial\Sigma)/ m \cL^\tau(\Sigma,\partial\Sigma).
\end{gather*}
When $\Sigma_1 \subset \Sigma_2$ is an inclusion of quasi-convex fans of the same dimension, we get naturally the inclusion 
\[ IH(\Sigma_1,\partial\Sigma_1) \to IH(\Sigma_2,\partial\Sigma_2),\]
and the surjection 
\[ IH(\Sigma_2) \to IH(\Sigma_1).\]

\subsection{Hard Lefschetz theorem}

Let $l\in\cA^2(\Sigma)$ be a piecewise linear function on $\Sigma$. Multiplication with $l$ defines a degree $2$ map on $\cL^\tau(\Sigma)$ and on $IH(\st\tau)$. The  following theorem was proved in \cite{Karu, BL2}.

\begin{theorem}[Hard Lefschetz] \label{thm-HL}
Let $\Sigma$ be a complete fan and $l\in\cA^2(\Sigma)$ strictly convex. Then multiplication with $l$ defines a \LO on $IH(\st\tau)$ for every $\tau\in\Sigma$:
\[ l^i : IH^{c-i}(\st\tau) \to IH^{c+i}(\st\tau)\]
is an isomorphism for every $i>0$, where $c=\codim\tau$.
\end{theorem}

Hard Lefschetz theorem gives rise to the Lefschetz decomposition of $IH(\st\tau)$. Let the primitive cohomology be 
\[ \prim_l IH^{c-i}(\st\tau) = \ker l^{i+1} : IH^{c-i}(\st\tau) \to IH^{c+i+2}(\st\tau).\]
Then
\[ IH(\st\tau) = \bigoplus_{i\geq 0} \bigoplus_{j=0}^{i} l^j \prim_l IH^{c-i}(\st\tau).\]

Hard Lefschetz theorem also gives the first part of the following corollary. The second part follows from the duality discussed below.

\begin{corollary} \label{cor-degrees}
For any $\tau < \sigma$ in $\Sigma$,
\begin{enumerate}
  \item $\cL^\tau_\sigma = \cL^\tau([\sigma])$ is a free graded $\cA_\sigma$-module generated in degrees $< \codim\tau-\codim\sigma$.
  \item $\cL^\tau([\sigma],\partial\sigma)$ is a free graded $\cA_\sigma$-module generated in degrees $> \codim\tau-\codim\sigma$.
\end{enumerate}
\end{corollary}

\begin{theorem} \label{thm-mor}
Let $\phi: \cL^\sigma\to \cL^\tau(j)$ be a nonzero morphism in $\fM$. Then $j\geq \codim\tau-\codim\sigma$. If the equality holds, then $\sigma=\tau$, $j=0$ and $\phi_\sigma$ is nonzero.  
\end{theorem}
\begin{proof}
Let $\nu\in\Sigma$ be a minimal cone for which $\phi_\nu$ is nonzero. Then $\phi_\nu$ gives a nonzero degree $j$ map
\[ \cL^\sigma([\nu]) \to \cL^\tau([\nu],\partial\nu).\]
By Corollary~\ref{cor-degrees}, $\cL^\sigma([\nu])$ has generators in degrees $\leq \codim\sigma - \codim\nu$, with equality only if $\sigma=\nu$, and $\cL^\tau([\nu],\partial\nu)$ has generators in degrees $\geq \codim\tau-\codim\nu$, with equality only if $\tau=\nu$. It follows that 
\[ j \geq (\codim\tau-\codim\nu)- (\codim\sigma - \codim\nu) = \codim\tau-\codim\sigma,\]
and equality holds only if $\sigma=\nu=\tau$.
\end{proof}

The theorem can be restated as follows. If $\phi: \cL^\sigma(\codim\sigma) \to \cL^\tau(\codim\tau+j)$ is a nonzero morphism in $\fM$ then $j\geq 0$, and equality holds only if $\sigma=\tau$ and $\phi_\sigma \neq 0$.

\begin{corollary}
\[ \Hom_\fM(\cL^\tau,\cL^\tau) = \Hom_A(\cL^\tau_\tau,\cL_\tau^\tau) = \RR.\]
\end{corollary}

We also mention a more precise form of the decomposition theorem.

\begin{lemma}\label{lem-prec-decomp}
Let $\cF$ in $\fM$ have a decomposition $\cF = \oplus_\tau W_\tau \otimes \cL^\tau$. Then
\[  W_\tau \isom \ker \big(\overline{\cF}_\tau \to \overline{\cF(\partial\tau)}\big)
  = \img \big(\overline{\cF([\tau],\partial\tau)} \to \overline{\cF}_\tau\big).\]
\end{lemma}

\begin{proof}
For any $\sigma\in\Sigma$ the map
\[ \overline{\cL^\tau_\sigma} \to \overline{\cL^\tau(\partial\sigma)} \]
is zero when $\sigma=\tau$ and an isomorphism otherwise. 
This gives the stated result.
\end{proof}

Note that the image of $W_\tau$ in  $\overline{\cF}_\tau$ is canonical, given by the lemma. However, the embedding $W_\tau\otimes \cL^\tau \subset \cF$ depends on the lifting of this subspace to an embedding of free $\cA_\tau$-modules.

As an example, when $\pi:\hat\Sigma\to\Sigma$ is a subdivision of fans, then in the decomposition $\pi_* \cL_{\hat\Sigma} = \oplus_\tau W_\tau \otimes \cL^\tau$ we have
\[ W_\tau \isom \img \big(IH(\hat\tau,\partial\hat\tau) \to IH(\hat\tau)\big),\]
where $\hat\tau \subset \hat\Sigma$ is the inverse image of $[\tau]$. When $\tau\in\Sigma$ is a maximal cone, then $W_\tau = \overline{W_\tau\tensor \cL^\tau(\Sigma)}$ is canonically a subspace of $IH(\hat\Sigma)$.

\subsection{Non-pointed fans} \label{sec-npfans}

The results listed above were proved in \cite{BBFK, BL1, BL2, Karu} for  pointed fans. However, it is not difficulty to extend these results to non-pointed fans.
We show some of the steps needed to reduce from the case of non-pointed fans to the case of pointed fans.

Let $\Sigma$ be a fan in $V$ and consider the fan $\Sigma\times\RR^m$ in $V\times \RR^m$, with projection $\pi: \Sigma\times\RR^m\to \Sigma$. Let $A$ be the ring of polynomial functions on $V$ and $B$ the ring of polynomial functions on $V\times\RR^m$. Define the pullback functor $\pi^*: \fM(\Sigma)\to \fM(\Sigma\times\RR^m)$ by 
\[ \pi^*(\cF) = \cF \otimes_A B.\]
It follows from the definitions that if $\cF$ lies in $\fM(\Sigma)$, then $\pi^*(\cF)$ lies in $\fM(\Sigma\times\RR^m)$. In particular, $\pi^*(\cL^\tau) = \cL^{\tau\times\RR^m}$. It is straight-forward to show that the decomposition theorem holds for non-pointed fans. This implies that the functor $\pi^*$ is essentially surjective. Since the operation $\cdot \otimes_A B$ is an exact functor on $A$-modules, we get for global sections
\[ \pi^*(\cF) (\Sigma\times\RR^m) = \cF(\Sigma)\otimes_A B,\]
and in case of a quasi-convex fan $\Sigma$, 
\[ \pi^*(\cF) (\Sigma\times\RR^m,\partial(\Sigma\times\RR^m)) = \cF(\Sigma,\partial\Sigma)\otimes_A B.\]

\begin{lemma}
Assume that $\Sigma$ is quasi-convex and let $\cF \in \fM(\Sigma)$. Then
\begin{enumerate}
  \item $\cF(\Sigma)$ is a free $A$-module.
  \item $\cF(\Sigma,\partial\Sigma)$ is a free $A$-module.
\end{enumerate}

\end{lemma}
\begin{proof}
Let $\Sigma=\overline\Sigma \times\RR^m$ for some pointed fan $\overline\Sigma$ and $\pi:\Sigma\to\overline{\Sigma}$ the projection. Then $\cF\isom \pi^*(\cG)$ for some $\cG\in \fM(\overline\Sigma)$. From the pointed case we know that $\cG(\overline{\Sigma})$ and $\cG(\overline{\Sigma},\partial\overline\Sigma)$ are free $\bar{A}$-modules, where $\bar{A}$ is the ring of polynomial function on the ambient space of $\overline{\Sigma}$. Now $\cF(\Sigma)$ and $\cF(\Sigma,\partial\Sigma)$  are obtained from these by applying the operation $ \cdot\otimes_{\bar{A}} A $, and hence are free $A$-modules.
\end{proof}

Consider again the fan $\Sigma\times\RR^m$. Note that for any $\cF$ in $\fM(\Sigma)$,
\begin{equation}\label{eq-ih} 
\overline{ \pi^* \cF(\Sigma \times\RR^m)} = \overline{\cF(\Sigma)},\end{equation}
and when $\Sigma$ is quasi-convex, then 
\[ \overline{ \pi^* \cF(\Sigma \times\RR^m, \partial(\Sigma\times\RR^m))} = \overline{\cF(\Sigma,\partial\Sigma)}.\]
Let $l\in\cA^2(\Sigma\times\RR^m)$ act on $\overline{ \pi^* \cF(\Sigma \times\RR^m)}$ by multiplication. We may change $l$ by a linear function in $B$ without changing the action. Thus, we may assume that $l$ vanishes on the minimal cone $\zz \times\RR^m$, and hence is the pullback of some function $\overline{l}$ in $\cA^2(\Sigma)$. Now the two functions act the same way on the spaces (\ref{eq-ih}). Thus we get:

\begin{lemma} \label{lem-HL-pointed}
Theorem~\ref{thm-HL} (HL) for pointed fans implies the same theorem for general fans.
\end{lemma}

Next let us consider subdivisions of fans and the functor $\pi_*$ in the non-pointed case.

\begin{lemma}
Let $\pi: \hat\Sigma \to \Sigma$ be a subdivision of fans. Then $\pi_*$ defines a functor from $\fM(\hat\Sigma)$ to $\fM(\Sigma)$.
\end{lemma}

\begin{proof}
Let $\cF \in \fM(\hat\Sigma)$. Then $\pi_*(\cF)$ is clearly a graded flabby sheaf of $\cA_\Sigma$-modules. We need to show that it is locally free. For $\sigma\in\Sigma$ let $\hat\sigma\subset \hat\Sigma$ be the inverse image of $[\sigma]$. Then $\pi_*(\cF)_\sigma = \cF(\hat\sigma)$. However, $\hat\sigma$ is a quasi-convex fan in $\Span\sigma$ and hence $\cF(\hat\sigma)$ is a free $\cA_\sigma$-module.
\end{proof}

The Hard Lefschetz theorem, Theorem~\ref{thm-HL}, is usually stated for the intersection cohomology $IH(\Sigma)$ only.
The case of $IH(\st\tau)$ can be reduced to it as follows. Given $\tau\in\Sigma$, define a new fan 
\[ \Sigma_\tau  = \{ \sigma+\Span\tau | \sigma\in \st\tau\}.\]
Then $\Sigma_\tau$ and $\st\tau$ are isomorphic as ringed spaces over $A$, hence $IH(\st\tau)= IH(\Sigma_\tau)$. If $\Sigma$ is complete, then so is $\Sigma_\tau$. Now the \HL theorem for $IH(\Sigma_\tau)$ implies the same theorem for $IH(\st\tau)$.

\section{Intersection pairing.}

In this section we study the \Po duality pairing on $IH(\Sigma)$ constructed in \cite{BL1, BL2}. We will need the pairing for non-pointed fans. For a non-pointed fan $\Sigma$, the intersection pairing on it can be reduced to the case of a pointed fan easily. The difficulty lies in comparing the pairings on a non-pointed fan and its pointed subdivision. To handle this case, we need to explain the definitions and constructions of Bresler-Lunts \cite{BL1, BL2}. We will rewrite some of the proofs and refer to \cite{BL1, BL2} whenever possible. An alternative construction of the duality functor and intersection pairing is given in \cite{BBFK2}.

\subsection{Intersection pairing for pointed fans.}

When $\Sigma$ is a complete simplicial fan of dimension $n$ in $V$, then Brion \cite{Brion} constructs a nondegenerate symmetric A-bilinear pairing 
\[ \cA(\Sigma)\times\cA(\Sigma) \to A(-2n).\]
The pairing is given by multiplication of piecewise polynomial functions, followed by a morphism $\cA(\Sigma)\to A(-2n)$, determined by a volume form $\Omega\in \Lambda^n V^*$. 
This pairing gives the \Po duality pairing
\[ IH(\Sigma)\times IH(\Sigma) \to \RR(-2n).\]
Bressler and Lunts \cite{BL1} extend this pairing to the spaces $\cL_\Sigma(\Sigma)$ and $IH(\Sigma)$ for nonsimplicial $\Sigma$. They first construct a contravariant duality functor
\[ \DD: \fM \to \fM.\]
If $\Sigma$ is complete, then for any $\cF$ in $\fM$ there is a natural isomorphism
\begin{equation} \label{eq-duality}
   \DD(\cF)(\Sigma) \to \Hom_A (F(\Sigma), \omega),
\end{equation} 
where $\omega = \det V^* \tensor A$ is the dualizing module of $A$. 
The sheaf $\DD(\cL_\Sigma)$ is canonically isomorphic to $\cL_\Sigma$. The isomorphism $\omega \isom A(-2n)$ then gives a nondegenerate $A$-bilinear pairing 
\[ \cL_\Sigma(\Sigma)\times\cL_\Sigma(\Sigma) \to A(-2n)\]
and the \Po duality pairing
\[ IH(\Sigma)\times IH(\Sigma) \to \RR(-2n).\]
The choice of the volume form $\Omega$ is used in this construction twice: once in the isomorphism $\omega \isom A(-2n)$, and second time in the construction of the isomorphism (\ref{eq-duality}) in the role of giving an orientation on $V$. As a consequence, when the form $\Omega$ is changed by a nonzero constant, the pairing changes by a positive constant.

\Po duality pairing on quasi-convex fans was studied in \cite{BBFK}. In \cite{BL2} this pairing was described in terms of the duality functor $\DD$. In that case the isomorphism similar to (\ref{eq-duality}) gives a nondegenerate $A$-bilinear pairing
\[ \cL_\Sigma(\Sigma,\partial\Sigma)\times\cL_\Sigma(\Sigma) \to A(-2n),\]
inducing the \Po duality pairing 
\[ IH(\Sigma,\partial\Sigma)\times IH(\Sigma) \to \RR(-2n).\]

The canonical pairing of Bressler and Lunts agrees with Brion's pairing (up to a factor of $n!$) in the case of simplicial fans. The pairing is also compatible with subdivisions of fans. If $\hat\Sigma\to\Sigma$ is a subdivision, we can use the decomposition theorem to embed $\cL_\Sigma(\Sigma) \subset \cL_{\hat\Sigma}(\hat\Sigma)$. Then the pairing on $\cL_\Sigma(\Sigma)$ constructed above is equal to the restriction of the pairing on $\cL_{\hat\Sigma}(\hat\Sigma)$. We may choose $\hat\Sigma$ simplicial, represent elements in $\cL_\Sigma(\Sigma)$ as piecewise polynomial functions on $\hat\Sigma$, and pair them using Brion's method. This in particular shows that the pairing is symmetric and $\cA(\Sigma)$-bilinear. Similar statements hold for quasi-convex fans. We mention two results that can be proved by reducing to multiplication of functions in the simplicial case:
\begin{enumerate}
  \item If $f$ and $g$ are sections of $\cL_\Sigma$ with disjoint supports, then the pairing between them is zero.
  \item Let $\Sigma_1\subset\Sigma_2$ be an inclusion of quasi-convex fans of the same dimension. Then the pairings on $\Sigma_1$ and $\Sigma_2$ are compatible. Let $f\in \cL_{\Sigma_1}(\Sigma_1, \partial \Sigma_1)$ and $g\in \cL_{\Sigma_2}(\Sigma_2)$. Then
\[ \langle f,g\rangle_{\Sigma_1} = \langle f,g\rangle_{\Sigma_2}.\] 
   
\end{enumerate}

\subsection{Injective and Projective sheaves}

We will be working with derived categories of sheaves and derived functors. The category of sheaves on $\Sigma$ has enough injectives and projectives. Let $\RR_{[\sigma]}$ be the constant sheaf with stalk $\RR$ on $[\sigma]$, extended by zero to a sheaf on $\Sigma$, and let $\RR_{\st\sigma}$ be defined similarly. The sheaves $\RR_{[\sigma]}$ are projective and the sheaves   $\RR_{\st\sigma}$ are injective. Any sheaf can be resolved using direct sums of these injective or projective sheaves. For an arbitrary sheaf $F$
\[ \Hom(\RR_{[\sigma]}, F) = F_\sigma, \qquad \Hom(F, \RR_{\st\sigma}) = F^*_\sigma.\]

\subsection{The cellular complex} \label{sec-cell}

The bilinear pairings in \cite{BL1,BL2} are constructed with the help of the cellular complex. We  recall the definition of this complex.

Let us fix an orientation on every cone $\sigma\in\Sigma$. The cellular complex $\Cd(F) = \Cd(\Sigma,F)$ of a sheaf $F$ on $\Sigma$ is the complex 
\[ 0\to C^0\to C^1\to \cdots \to C^n\to 0,\]
where $C^i = \oplus_{\codim\sigma=i} F_\sigma$, and the map between $F_\sigma$ and $F_\tau$ in the complex is $\pm \res^\sigma_\tau$, where the sign depends on whether the orientations of $\sigma$ and $\tau$ agree or not.

Note that $\Cd(\Sigma, \cdot)$ defines a functor from the category of sheaves on $\Sigma$ to the category of complexes of vector spaces. This functor is exact. We extend it to a functor from the bounded derived category of sheaves on $\Sigma$ to the bounded derived category of complexes. 

Let $\Sigma$ be complete. Fix an orientation on $V$ and define the map
\begin{align*} 
\Gamma(\Sigma,F) &\to H^0(\Cd(\Sigma,F))\\
(f_\sigma)_\sigma &\mapsto (\pm f_\sigma)_{\codim\sigma=0},
\end{align*}
where the sign depends on whether the orientations of $\sigma$ and $V$ agree or not. It is easy to see that this map is an isomorphism. Moreover, since $\Cd$ is an exact functor, the map extends to a morphism in the derived category of complexes
\[ R\Gamma(F)\to \Cd(F).\]

\begin{lemma} When $\Sigma$ is complete then the map $R\Gamma(F)\to \Cd(F)$ is an isomorphism, natural in $F$.
\end{lemma}

\begin{proof}
Consider the projective resolution $P^\bullet$ of the constant sheaf $\RR_\Sigma$:
\[ \cdots \to \oplus_{\codim \tau = 1} \RR_{[\tau]} \to \oplus_{\codim \sigma = 0} \RR_{[\sigma]} \to \RR_\Sigma,\]
where the maps $\RR_{[\tau]} \to \RR_{[\sigma]}$ in the complex are given by $\pm \Id_\RR$, with sign depending on the orientations of $\sigma$ and $\tau$, and the maps $\RR_{[\sigma]} \to \RR_\Sigma$ are similarly defined by $\pm \Id_\RR$, with sign depending on the orientations of $\sigma$ and $V$. Now
\[ R\Gamma(F) = R\Hom(\RR_\Sigma, F) \isom \Hom(P^\bullet, F),\]
and the last complex is equal to $\Cd(F)$.
\end{proof}

When $(\Sigma, \partial\Sigma)$ is a quasi-convex fan, then a similar map defines an isomorphims
\[ F(\Sigma,\partial\Sigma) \to H^0(\Cd(F)),\]
that extends to an isomorphism
\[ R\Gamma_c (F) \to \Cd(F).\]
Here $\Gamma_c$ is the functor $\Gamma_c(F) = F(\Sigma,\partial\Sigma)$. 
To get the global sections functor on a quasi-convex fan, we need to index the cellular complex over $\Int\Sigma = \Sigma\setmin\partial\Sigma$. Then the isomorphism
\[ F(\Sigma) \to H^0(\Cd(\Int\Sigma, F))\]
extends to an isomorphism $R\Gamma(\Sigma, F) \isom \Cd(\Int\Sigma, F)$. 

\subsection{The cellular complex and the push-forward map}

Let $\pi:\hat\Sigma\to\Sigma$ be a subdivision. We fix orientations on all cones in $\hat\Sigma$ and $\Sigma$. For any sheaf $F$ on $\hat\Sigma$ there is a natural morphism of complexes $\Cd(\pi_* F)\to \Cd(F)$, defined on the summand $(\pi_*F)_\sigma$  as 
\[ (f_{\sigma'})_{\pi(\sigma')\in[\sigma]} \longmapsto (\pm f_{\sigma'})_{\dim\sigma'=\dim\sigma},\]
where the sign depends on the orientations of $\sigma$ and $\sigma'$. It is easy to check that this indeed gives a morphism of complexes. Replacing $F$ by its injective resolution, the morphism extends to 
\[ \Cd(R \pi_* F) \to \Cd(F).\]
 
 \begin{lemma} \label{lem-cell}
 The morphism $\Cd(R \pi_* F) \to \Cd(F)$ is a quasi-isomorphism. If $\Sigma$ is quasi-convex, then the following diagram of quasi-isomorphisms commutes
 \[ \begin{CD} R\Gamma_c(R\pi_* F) @>>> R\Gamma_c(F) \\
@VVV @VVV\\
\Cd( R\pi_* F) @>>> \Cd(F) \\
\end{CD} \]  
 \end{lemma}

\begin{proof}
It suffices to prove the statement of the lemma for injective sheaves $F$ of the form $F=\RR_{\st\tau}$. We may replace $\hat\Sigma$ with $\st\tau$, $\Sigma$ with $\st\pi(\tau)$ (or, more precisely, replace $\hat\Sigma$ with $\hat\Sigma_\tau$ and $\Sigma$ with $\Sigma_{\pi(\tau)}$ as at the end of Section~\ref{sec-npfans}), and thus assume that $F=\RR_{\hat\Sigma}$, $\pi_* F = \RR_\Sigma$. 

The quasi-isomorphism of 
\[ \Cd(\RR_\Sigma) \to \Cd(\RR_{\hat\Sigma})\]
follows from results in algebraic topology. Let $\Delta$ be the intersection of $|\Sigma|$ with an $(n-1)$-sphere in $V$ centered at the origin. Then $\Sigma$ and $\hat\Sigma$ give $\Delta$ two structures of a CW-complex. The cellular complexes $\Cd(\RR_\Sigma)$ and $\Cd(\RR_{\hat\Sigma})$ compute the reduced cellular homology of $\Delta$, and the map is the map coming from refinement. Since the cellular homology only depends on the topological space, the map is a quasi-isomorphism.  

To prove the commutativity of the diagram, again we may assume that $F=\RR_{\hat\Sigma}$, $\pi_* F = \RR_\Sigma$. All complexes in the diagram are quasi-isomorphic to zero unless $\partial\Sigma=\emptyset$, that means, $\Sigma$ and $\hat\Sigma$ are complete fans. In that case each complex is naturally quasi-isomorphic to $\RR$ and the maps between them are the identity maps $\Id_\RR$.
\end{proof}

\subsection{Categories of modules}

Let $A$-mod be the category of graded finitely generated $A$-modules, and  
$A_\Sigma$-mod the category of sheaves of graded $A$-modules on $\Sigma$ of finite type. We write $\Hom(F,G)$ for graded morphisms and 
\[ \Homd(F,G) = \oplus_{j\in\ZZ} \Hom (F,G(j)).\]
If $F$ and $G$ are (sheaves of) $A$-modules, then $\Homd(F,G)$ is an $A$-module.

The category $A_\Sigma$-mod has enough projectives. Let $A_{[\sigma]}$ be the constant sheaf with stalk $A$ on $[\sigma]$, extended by zero to $\Sigma$. Such sheaves $A_{[\sigma]}$ for $\sigma\in\Sigma$ and their shifts are projective. Any sheaf in $A_\Sigma$-mod is quasi-isomorphic to a finite length complex of such projectives.

Let $A_{\st\sigma}$ be the constant sheaf with stalk $A$ on $\st\sigma$, extended by zero to $\Sigma$. Such sheaves and their shifts are injective as sheaves of vector spaces. Note that 
\[ \Hom_{A_\Sigma\text{-mod}} (F, A_{\st\sigma}) = \Hom_{A\text{-mod}}(F_\sigma, A).\]
Hence the sheaf $A_{\st\sigma}$ is injective in the full subcategory of $A_\Sigma$-mod consisting of sheaves with all stalks free $A$-modules. Every sheaf in this subcategory is quasi-isomorphic to a finite length complex of such injective sheaves. Combining the projective and injective resolutions, we see that every sheaf in $A_\Sigma$-mod is quasi-isomorphic to a finite length complex of injective sheaves $A_{\st\sigma}$ and their shifts.

We use the injective and projective resolutions to compute derived functors. As an example, when $F,G$ are in $A_\Sigma$-mod, then we have quasi-isomorphisms of complexes of $A$-modules.
\[ R\Homd (F,G) \isom \Homd(P,G) \isom \Homd(F',I),\]
where $P$ is a projective resolution of $F$, $I$ is an injective resolution of $G$, and $F'$ is a resolution of $F$ in sheaves with all stalks free $A$-modules.

Let $D^b(A_\Sigma\text{-mod})$ be the bounded derived category of $A_\Sigma$-mod.
We consider the category $\fM(\Sigma)$ as a full subcategory of the category $A_\Sigma$-mod and of the derived category $D^b(A_\Sigma\text{-mod})$.

The cellular complex defines a functor from $A_\Sigma$-mod to complexes in $A$-mod. This functor is exact and extends to a functor between derived categories
\[ \Cd: D^b(A_\Sigma\text{-mod}) \to D^b(A\text{-mod}).\]

When $\pi$ is a subdivision, then the quasi-isomorphism of cellular complexes $\Cd(R \pi_* F) \to \Cd(F)$ extends to an isomorphism of functors $\Cd(R \pi_* \cdot) \to \Cd(\cdot)$. Here the two functors can be viewed as functors $D^b(A_\Sigma\text{-mod}) \to D^b(A\text{-mod})$.

 \subsection{The functor $\tilde{\DD}$.}
 
 We will now define the duality functor $\tilde{\DD}$ for not necessarily pointed fans. This differs from the functor $\DD$ in \cite{BL1} by a shift in grading.
 
 Let $A_{\{\zz\}}$ be the sheaf on $\Sigma$ with stalk $A$ at $\zz$ and zero at other cones. Define the dualizing complex $\tilde{D}_\Sigma = A_{\{\zz\}}[\codim\zz]$ in $D^b(A_\Sigma\text{-mod})$. This is a complex of sheaves supported on $\zz$ and concentrated in cohomology degree $-\codim\zz$. The contravariant duality functor $\tilde\DD = \tilde\DD_\Sigma: D^b(A_\Sigma\text{-mod})\to D^b(A_\Sigma\text{-mod})$  is defined as the sheaf hom into the dualizing complex:
 \[ \tilde\DD(F) = R\cHomd(F, \tilde{D}_\Sigma).\]
In \cite{BL1} the duality functor $\DD$ was defined similarly, but with $\tilde{D}_\Sigma$ replaced by $D_\Sigma = \tilde{D}_\Sigma \otimes \det V^*$. It follows that 
\[ \DD(F) = \tilde{\DD}(F) \otimes \det V^* \isom \tilde{\DD}(F) (-2\codim\zz).\]

Let us give an actual computation of the dual sheaf. Consider the injective resolution of the dualizing complex $\tilde{D}_\Sigma$,
\[ I = (0\to A_{\st\zz} \to \bigoplus_{\dim\rho=\dim\zz+1} A_{\st\rho} \to \bigoplus_{\dim\tau=\dim\zz+2} A_{\st\tau} \to \cdots),\]
where the maps between $A_{\st\rho}$ and $A_{\st\tau}$ are defined by $\pm \Id_A$, with the sign depending on the orientations of $\rho$ and $\tau$. The term $A_{\st\zz}$ lies in cohomology degree $-\codim\zz$. If $F$ is a sheaf in $A$-mod with all stalks free $A$-modules, then $\tilde\DD(F)$ is the complex of sheaves $\cHomd(F,I)$. When $F$ is an arbitrary complex in $D^b(A_\Sigma\text{-mod})$, we first replace it with a complex where all sheaves have stalks free $A$-modules, and then apply $\cHomd( \cdot,I)$ to this complex.

\begin{proposition} \label{prop-dual}
For $F$ in $D^b(A_\Sigma\text{-mod})$ there is a natural isomorphism in $D^b(A\text{-mod})$
\[ \psi_\Sigma: R\Gamma(\tilde\DD(F)) \stackrel{\isom}{\longrightarrow} R\Homd_A(\Cd(F), A).\]
\end{proposition}

\begin{proof}

It suffices to construct the isomorphism $\psi_\Sigma$ canonically for sheaves $F$ with all stalks free $A$-modules. For such $F$ we have that
\[ R\Gamma(\tilde\DD(F)) = R\Homd (F, \tilde{D}_\Sigma) = \Homd (F, I),\]
which is equal to the complex
\[ 0\to F_\zz^* \to \bigoplus_{\dim\rho=\dim\zz+1} F_\rho^* \to \bigoplus_{\dim\tau=\dim\zz+2} F_\tau^* \to \cdots ,\]
where for a free graded $A$-module $M$ we denote $M^* = \Homd_A(M,A)$. The term $F_\zz^*$ of the complex lies in cohomology degree $-\codim\zz$. This complex is equal to $\Cd(F)^*$.
\end{proof}

When $\Sigma$ is a complete fan, then a choice of orientation on $V$ defines an isomorphism $R\Gamma(F) \isom \Cd(F)$, hence the proposition gives a nondegenerate $A$-bilinear pairing
\[ R\Gamma(\tilde{\DD}(F)) \times R\Gamma(F) \to A.\]
Similarly, for quasiconvex $\Sigma$ we get the pairing
\[  R\Gamma(\tilde{\DD}(F)) \times R\Gamma_c (F) \to A.\]

\subsection{Compatibility with $\pi^*$.}

Consider the fan $\Sigma\times \RR^m$ in $V\times \RR^m$ and the projection morphism $\pi: \Sigma\times \RR^m \to \Sigma$. Let $A$ be the ring of polynomial functions on $V$ and $B$ the ring of polynomial functions on $V\times\RR^n$. Fix an orientation on $\RR^m$. Then an orientation on $\sigma$ gives an orientation on $\sigma\times\RR^m$.

Define the functor $\pi^*: A_\Sigma\text{-mod} \to B_{\Sigma\times\RR^m}\text{-mod}$ by
\[ \pi^*(F) = F \otimes_A B.\]
The functor $\pi^*$ is exact and extends to a functor on derived categories. Also because of exactness of the operation $\cdot \otimes_A B$, there are natural isomorphisms
\begin{align*}
  R\Gamma(\pi^* F) &\isom R\Gamma(F) \otimes_A B,\\
  \Cd (\pi^* F) &\isom \Cd(F) \otimes_A B. 
\end{align*}

\begin{lemma} \label{lem-compat1}
\begin{enumerate}
  \item The functor $\pi^*: D^b(A_\Sigma\text{-mod})\to D^b(B_{\Sigma\times\RR^m}\text{-mod})$ commutes with $\tilde \DD$:
  \[ \tilde{\DD} \circ \pi^* = \pi^* \circ\tilde\DD.\]
  
  \item The isomorphism $\psi_\Sigma$ in Proposition~\ref{prop-dual} is compatible with $\pi^*$ in the sense that the following diagram commutes
\[ \begin{CD} R\Gamma(\tilde\DD(\pi^* F)) @>{\psi_{\Sigma\times\RR^m}}>> R\Homd_B(\Cd(\pi^* F), B) \\
@V{\isom}VV @V{\isom}VV\\
R\Gamma(\tilde\DD(F))  \otimes_A B @>{\psi_\Sigma\otimes \Id}>> R\Homd_A(\Cd(F), A) \otimes_A B\\
\end{CD} \]  
\end{enumerate}
\end{lemma}

\begin{proof}
If $M$ and $N$ are graded $A$\text{-mod}ules, then 
\[ \Homd_B(M\otimes_A B, N\otimes_A B) = \Homd_A(M,N)\otimes_A B.\]
Note that $\tilde{D}_{\Sigma\times\RR^m} = \tilde{D}_\Sigma \otimes_A B$. This implies that for any $F$ in $D^b(A_\Sigma\text{-mod})$,
\[ R\cHomd (F\otimes_A B, \tilde{D}_\Sigma \otimes_A B) = R\cHomd (F, \tilde{D}_\Sigma)\otimes_A B.\]

To prove the second part of the lemma we need the construction of the isomorphism $\psi_\Sigma$ from the proof of Proposition~\ref{prop-dual}. Recall that if $F$ is in $A_\Sigma\text{-mod}$ with all stalks free $A$-modules, then 
\[ R\Gamma(\tilde\DD(F)) = \Homd (F, I) = \Cd(F)^*,\]
where $I$ is the injective resolution of $\tilde{D}_\Sigma$. Applying the operation $\otimes_A B$ to all terms of the equation gives the construction of $\psi_{\Sigma\times\RR^m}$  for $\pi^* F$.

For general $F$ in $D^b(A_\Sigma\text{-mod})$ we first replace it with a complex of sheaves with all stalks free $A$-modules (for example, a projective resolution), and then apply the previous argument to each sheaf in the complex.
\end{proof}

The lemma implies that when $\Sigma$ is quasi-convex, then the pairing 
\[  R\Gamma(\tilde{\DD}(\pi^* F)) \times R\Gamma_c (\pi^* F) \to B\]
constructed on $\Sigma\times\RR^m$ is obtained from the pairing 
\[  R\Gamma(\tilde{\DD}(F)) \times R\Gamma_c (F) \to A\]
on $\Sigma$ by applying the operation $\otimes_A B$ to all terms.

\subsection{Duality in $\fM(\Sigma)$}

 As was explained in Section~\ref{sec-npfans}, $\pi^*$ restricts to an essentially surjective functor $\fM(\Sigma)\to \fM(\Sigma\times\RR^m)$.

\begin{lemma}
For any $\Sigma$ the duality functor $\tilde\DD$ maps $\fM(\Sigma)$ to $\fM(\Sigma)$. Moreover, $\tilde\DD \circ \tilde{\DD}$ is naturally isomorphic to the identity functor. 
\end{lemma}

\begin{proof}
Write $\Sigma = \overline{\Sigma}\times\RR^m$ for some pointed fan $\overline{\Sigma}$. Let $\pi: \Sigma\to\overline{\Sigma}$ be the projection. Since the functor $\tilde{\DD}$ is a shift of the functor $\DD$, we know the lemma for the pointed fan $\overline{\Sigma}$ from \cite{BL1}. Now part (1) of the previous lemma proves it for $\Sigma$.
\end{proof}

By the lemma, the functor $\tilde\DD$ maps indecomposable sheaves in $\fM$ to indecomposable sheaves. Since the functor is local, it maps $\cL^\tau$ to $\cL^\tau(j)$
for some $j$. More precisely:

\begin{lemma} \label{lem-dualL}
\[ \tilde{\DD} (\cL^\tau) = \big( \det \tau^\perp\big)^* \otimes \cL^\tau \isom \cL^\tau(2\codim\tau).\]
\end{lemma}

\begin{proof}
We need to compute the stalk of $\tilde{\DD} (\cL^\tau)$ at $\tau$. Since the functor $\tilde\DD$ is local, we may replace $\Sigma$ with $[\tau]$. Then $\cL^\tau$ is the sheaf with stalk $\cA_\tau$ at $\tau$ and zero everywhere else. Moreover,
\[ \tilde{\DD} (\cL^\tau)_\tau = \Gamma(\tilde{\DD} (\cL^\tau)) = R^0\Homd ( \Cd(\cL^\tau), A)\]
by Proposition~\ref{prop-dual}. The last space is $\Ext^{\codim\tau}_A(\cA_\tau,A)$. Let us use the Koszul resolution of $\cA_\tau$:
\[ 0\to \det \tau^\perp \otimes A \to \cdots \to \Lambda^2 \tau^\perp \otimes A \to \Lambda^1 \tau^\perp \otimes A \to A \to \cA_\tau.\]
From this we get that the Ext group is $\big( \det \tau^\perp\big)^* \otimes \cA_\tau$. 
\end{proof}

Let us fix a volume form $\Omega_\tau\in \det \tau^\perp$, giving an isomorphism of $\cA_\Sigma$-modules $\phi: \cL^\tau \to \DD(\cL^\tau)(-2\codim\tau)$. Then $\phi$ is symmetric in the sense that $\tilde{\DD}(\phi) = \phi(2\codim\tau)$. 

Let $\Sigma$ be quasi-convex. Note that the volume form $\Omega_\tau$ determines an orientation on $V/\Span(\tau)$. Together with the fixed orientation on $\tau$ they determine an orientation on $V$ (in the order orientation of $\tau$ times orientation on $V/\Span\tau$) and hence an isomorphism $H^0(\Cd(\cL^\tau)) \isom \cL^\tau(\Sigma,\partial\Sigma)$. Proposition~\ref{prop-dual} and Lemma~\ref{lem-dualL} then give a nondegenerate $A$-bilinear pairing
\[ \cL^\tau(\Sigma) \times \cL^\tau(\Sigma,\partial\Sigma) \to A(-2\codim\tau)\]
and the $\RR$-linear \Po duality pairing
\[ IH(\st\tau) \times IH(\st\tau, \partial\st\tau) \to \RR(-2\codim\tau).\]
These pairings are $\cA(\Sigma)$-bilinear in the sense that for $f\in \cA(\Sigma)$,
\[ \langle f\cdot,\cdot\rangle = \langle\cdot,f\cdot\rangle.\]
Moreover, the \Po pairing is symmetric when restricted to 
\[ IH(\st\tau,\partial\st\tau) \times IH(\st\tau, \partial\st\tau) \to \RR(-2\codim\tau),\]
and similarly for $\cL^\tau(\Sigma,\partial\Sigma)$.

Define the fan $\overline{\st\tau}$ as the image of $\st\tau$ in $V/\Span\tau$. This is a quasi-convex pointed fan with boundary $\partial \overline{\st\tau}$. Then $\pi^*$ induces isomorphisms
\begin{align*}
  IH(\overline{\st\tau}) &\isom IH(\st\tau),\\
  IH(\overline{\st\tau},\partial\overline{\st\tau} ) &\isom IH(\st\tau,\partial\st\tau).
\end{align*}
Lemma~\ref{lem-compat1} implies that the pairing on $\st\tau$ is the same as the pairing on $\overline{\st\tau}$. The latter pairing is the same as the pairing of Bressler and Lunts \cite{BL1, BL2}.

\subsection{Compatibility with subdivisions}

Let $\pi:\hat\Sigma\to\Sigma$ be a subdivision. Recall that in Section~\ref{sec-cell} we constructed for any sheaf $F$ on $\hat\Sigma$ a natural quasi-isomorphism of complexes $\Cd(R \pi_* F)\to \Cd(F)$. When $F$ is in $D^b(A_{\hat\Sigma}\text{-mod})$, this induces an isomorphism in $D^b(A\text{-mod})$
\[ R\Homd_A(\Cd(F), A) \to R\Homd_A(\Cd(R\pi_* F), A).\]

\begin{theorem} \label{thm-compat2}
Let $\pi: \hat\Sigma\to \Sigma$ be a subdivision and let $F$ be in $D^b(A_{\hat\Sigma}\text{-mod})$. 
\begin{enumerate}
   \item There is a natural isomorphism in $D^b(A_{\hat\Sigma}\text{-mod})$
  \[ R\pi_* \tilde{\DD} (F) \isom \tilde{\DD}( R\pi_* F).\]
  \item The isomorphism $\psi_\Sigma$ in Proposition~\ref{prop-dual} is compatible with $\pi_*$ in the sense that the following diagram commutes:
  \[ \begin{CD} R\Gamma(\tilde\DD(F)) @>{\psi_{\hat\Sigma}}>> R\Homd_A(\Cd(F), A) \\
@V{\isom}VV @V{\isom}VV\\
R\Gamma(\tilde\DD( R\pi_* F)) @>{\psi_\Sigma}>> R\Homd_A(\Cd(R\pi_* F), A)\\
\end{CD} \]  
  
\end{enumerate}
\end{theorem}

\begin{proof}
We construct the isomorphism in part (1) canonically for injective sheaves $F = A_{\st\tau}$, for $\tau\in\hat\Sigma$. To construct the isomorphism of stalks at $\sigma\in\Sigma$ we replace $\Sigma$ with $[\sigma]$, $\hat\Sigma$ with $\pi^{-1}[\sigma]$ and use the diagram in part (2) to define the left vertical arrow:
\[ R\pi_* \tilde{\DD} (F)_\sigma = R\Gamma(\pi^{-1}[\sigma], \tilde\DD(F)) \longrightarrow R\Gamma ([\sigma], \tilde\DD(\pi_* F)) = \tilde{\DD}( R\pi_* F)_\sigma.\]
(It is important here that the morphism $\psi_{\Sigma}$ is defined as an equality of complexes, hence it has an inverse.)

To see that these isomorphisms of stalks give isomorphisms of sheaves, note that for $\sigma \leq \nu$ there is an inclusion of complexes $\Cd([\sigma],\pi_* F)\to \Cd([\nu],\pi_* F)$, and a similar inclusion $\Cd(\pi^{-1}[\sigma], F)\to \Cd(\pi^{-1}[\nu],F)$. The mophism of complexes $\Cd(\pi_* F)\to \Cd(F)$ commutes with these inclusions. Dualizing the inclusion gives the surjection
\[ R\Homd_A(\Cd([\nu],\pi_* F), A) \to R\Homd_A(\Cd([\sigma],\pi_* F), A),\]
and a similar surjection on $\pi^{-1}[\nu]$ and $\pi^{-1}[\sigma]$. Via the isomorphisms $\psi_\Sigma$ and $\psi_{\hat\Sigma}$ the surjections correspond to restriction maps of sheaves. This implies that the isomorphism of stalks defined above gives an isomorphism of sheaves.

The same argument as in the previous paragraph proves the commutativity of the diagram in part (2). We need that the restriction map from $\Sigma$ to $[\sigma]$ (instead of from $[\nu]$ to $[\sigma]$) is compatible with $\psi_\Sigma$, $\psi_{\hat\Sigma}$ and the map of complexes $\Cd(\pi_* F)\to \Cd(F)$.
\end{proof}

Theorem~\ref{thm-compat2} and Lemma~\ref{lem-cell}  imply that if $\Sigma$ is  quasi-convex, then the pairing 
\[  R\Gamma(\tilde{\DD}(F)) \times R\Gamma_c (F) \to A\]
constructed on $\hat\Sigma$ is the same as the pairing
\[  R\Gamma(\tilde{\DD}(\pi_* F)) \times R\Gamma_c (\pi_* F) \to A\]
constructed on $\Sigma$.

\subsection{Hodge-Riemann bilinear relations.}
Let $\Sigma$ be a complete fan and $\tau\in\Sigma$, $c=\codim\tau$. We will use the \Po pairing $\langle\cdot,\cdot\rangle$ on $IH(\st\tau)$. For $l\in\cA^2(\Sigma)$, define the quadratic form $Q_l$ on $IH^{c-i}(\st\tau)$:
\[ Q_l(h) = \langle l^i h, h\rangle.\]

\begin{theorem}[Hodge-Riemann bilinear relations] \label{thm-HR}
Assume that $l$ is strictly convex on $\st\tau$. Then the quadratic form 
\[ (-1)^{\frac{c-i}{2}} Q_l \]
is positive definite of $\prim_l IH^{c-i}(\st\tau)$. 
\end{theorem}

The theorem is proved in \cite{Karu,BL2} for pointed fans $\Sigma$ and intersection cohomology $IH(\Sigma)$. The case of nonpointed fans and spaces $IH(\st\tau)$ can be reduced to this the same way as in Section~\ref{sec-npfans}. 

An equivalent statement of Theorem~\ref{thm-HR} is that for any $i\leq c/2$ the quadratic form $Q_l$ on $IH^{2i}(\st\tau)$ is nondegenerate and has signature
\[ \sum_{j<i} (-1)^j 2\dim IH^{2j}(\st\tau) + (-1)^i \dim IH^{2i}(\st\tau).\]
Hodge-Riemann bilinear relations imply the Hard Lefschetz theorem. The latter is equivalent to $Q_l$ being nondegenerate on $IH^{2i}(\st\tau)$.

\subsection{Products of fans.} \label{sec-products}

We list some results about products of fans that are needed in the proofs below. Since we need these results for pointed fans only, we refer to \cite{BL2} for proofs.

All fans in this subsection are assumed to be pointed.

\begin{lemma}[K\"unneth formula]
Let $\Sigma_i$ be a quasiconvex fan in $V_i$ for $i=1,2$, and $\pi_i: \Sigma_1\times\Sigma_2\to \Sigma_i$ the projection.
\begin{enumerate}
  \item There are natural $\cA(\Sigma_1)\otimes\cA(\Sigma_2)$-linear isomorphisms
  \begin{align*}
    IH(\Sigma_1\times\Sigma_2) &\isom IH(\Sigma_1)\otimes IH(\Sigma_2), \\
    IH(\Sigma_1\times\Sigma_2, \partial (\Sigma_1\times\Sigma_2)) &\isom IH(\Sigma_1, \partial\Sigma_1)\otimes IH(\Sigma_2, \partial\Sigma_2).
      \end{align*}
    The \Po duality pairing on the product $\Sigma_1\times\Sigma_2$ is induced by the pairing on each factor.
  \item If $\Sigma_i$ is complete and $l_i$ satisfies HL on $IH(\Sigma_i)$ for $i=1,2$, then $\pi_1^*l_1+\pi_2^* l_2 = l_1\otimes 1 + 1\otimes l_2$ satisfies HL on $IH(\Sigma_1\times\Sigma_2)$.
  \item If $\Sigma_i$ is complete and $l_i$ satisfies HR bilinear relations on $IH(\Sigma_i)$ for $i=1,2$, then $\pi_1^*l_1+\pi_2^* l_2$ satisfies HR bilinear relations on $IH(\Sigma_1\times\Sigma_2)$.
\end{enumerate}
\end{lemma}

Let $\Sigma$ be a fan in $V$ and $\tau\in\Sigma$. We say that $\Sigma$ has a {\em local product structure} at $\tau$ if every cone $\sigma\in\st\tau$ can be written as $\sigma = \tau+\sigma'$ for some $\sigma'\leq \sigma$ such that $\Span\sigma' \cap \Span\tau = 0$. 
Let $\overline{\st\tau}$ be the image of $\st\tau$ in $V/\Span \tau$.

When $\Sigma$ has a local product structure at $\tau$ then we can compare $[\st\tau]$ in $\Sigma$ and in $[\tau]\times \overline{\st\tau}$. The two are isomorphic as ringed spaces over $A$, hence they have isomorphic interesection cohomology.

\begin{lemma}
Let $\Sigma$ have a local product structure at $\tau$. Then we have an $\cA(\overline{\st\tau})$-linear isomorphism
\[ IH([\st\tau]) = IH(\st\tau) \isom IH(\overline{\st\tau}) \otimes \overline{\cL_\tau}.\]
\end{lemma} 

Consider the action of $l\in\cA^2(\Sigma)$ on $IH([\st\tau])$. We may change $l$ by a global linear function and assume that it vanishes on $\tau$. Hence $l$ is pulled back from $\overline{\st\tau}$ and it acts on $IH(\overline{\st\tau}) \otimes \overline{\cL_\tau}$ by acting on the first factor.

\begin{lemma} \label{lem-loc-prod}
Let $\Sigma$ have a local product structure at $\tau$ and let $l\in\cA^2(\Sigma)$. Assume that $\overline{\st\tau}$ is complete and $l$ defines a \LO on $IH(\overline{\st\tau})$ (for example, when $l$ is strictly convex). Then 
\[ l^i: IH^{n+j-i}([\st\tau],\partial[\st\tau]) \to IH^{n+j+i}([\st\tau],\partial[\st\tau])\]
is injective for any $j\leq \dim\tau$, $n=\dim\Sigma$.
\end{lemma}

\begin{proof}
The claim is equivalent to the dual map of $l^i$ being surjective. The dual of $l^i$ is
\[ l^i: IH^{n-j-i}([\st\tau]) \to IH^{n-j+i}([\st\tau]).\]
Since $l$ acts on the first factor of $IH(\overline{\st\tau}) \otimes \overline{\cL_\tau}$ and $\overline{\cL_\tau}$ lies in non-negative degrees, this is equivalent to 
 \[ l^i: IH^{n-j-i}(\overline{\st\tau}) \to IH^{n-j+i}(\overline{\st\tau})\]
being surjective, which follows from the Hard Lefschetz property because $n-j \geq \dim \overline{\st\tau}$.
\end{proof}

Let $\tau\in\Sigma$ be simplicial. Choose a generator $\chi_\tau$ for $\cA([\tau],\partial\tau)$ such that $\chi_\tau$ is positive in the interior of $\tau$.

\begin{lemma} \label{lem-loc-mult}
  Let $\Sigma$ have a local product structure at a simplicial cone $\tau$. Assume that $\overline{\st\tau}$ is complete.
\begin{enumerate}
  \item We have an $\cA(\overline{\st\tau})$-linear isomorphism
  \[ IH(\overline{\st\tau}) \isom IH([\st\tau]).\]
\item Multiplication with $\chi_\tau$ defines an $\cA([\st\tau])$-linear isomorphism
\[  \chi_\tau: IH([\st\tau]) \to IH([\st\tau],\partial [\st\tau]).\]
\item Combining the two isomorphisms above, the \Po duality pairings on the quasi-convex fan $[\st\tau]$ and the complete fan $\overline{\st\tau}$ agree in the sense that
\[ \langle\cdot,\cdot\rangle_{\overline{\st\tau}} = \langle\chi_\tau \cdot,\cdot\rangle_{[\st\tau]}\]
up to a positive constant factor that depends on $\chi_\tau$.
\end{enumerate}
\end{lemma}

\section{The filtration $\tau_{\leq p}$.} \label{sec-tau}

In this section we define the analogue of the perverse filtration for sheaves $\cF$ in $\fM = \fM(\Sigma)$:
\[ \cdots \subseteq \tau_{\leq p-1}\cF  \subseteq \tau_{\leq p}\cF \subseteq \tau_{\leq p+1}\cF \subseteq \cdots \subseteq\cF,  \]
and study its compatibility with morphisms in $\fM$ and the intersection pairing. As was mentioned in the introduction, we take $\cL^\sigma(\codim\sigma)$ as our basic indecomposable sheaves in $\fM$ and study decompositions 
\[ \cF = \oplus_\sigma W_\sigma \otimes \cL^\sigma(\codim\sigma).\]
Of particular interest to us is the decomposition of the sheaf $\cF = \pi_*\big(\cL_{\hat\Sigma}(\codim\hat\zz)\big)$ for a subdivision $\pi: \hat\Sigma\to \Sigma$. Here and in the following we denote $\zz = \zz_\Sigma$ and $\hat\zz = \zz_{\hat\Sigma}$.

\subsection{Definition of $\tau_{\leq p}$.}
For any $\cF$ in $\fM$, fix a decomposition
\[ \cF = \oplus_\sigma W_\sigma\tensor \cL^\sigma(\codim\sigma),  \qquad W_\sigma = \oplus_i W_\sigma^i,\]
and define for $p\in\ZZ$
\[ \tau_{\leq p} \cF = \bigoplus_\sigma \bigoplus_{i\leq p} W_\sigma^i \tensor \cL^\sigma(\codim\sigma).\]

\begin{proposition}
Let $\cF, \cG$ be sheaves in $\fM$ with fixed decompositions. Then any morphism $\cF \to \cG$ in $\fM$ maps $\tau_{\leq p} \cF$ to $\tau_{\leq p} \cG$.
\end{proposition}
\begin{proof}
This follows immediately from Theorem~\ref{thm-mor}.
\end{proof}

\begin{corollary}
The subsheaf $\tau_{\leq p} \cF \subseteq \cF$ is independent of the chosen decomposition.
\end{corollary}
\begin{proof}
If we have two subsheaves $\tau_{\leq p} \cF$ and $\tau'_{\leq p} \cF$ defined by two different decompositions, then the identity morphism $id:\cF\to \cF$ maps one to the other.
\end{proof}

Let $\fM^{\leq p}$ be the full subcategory of $\fM$ consisting of sheaves isomorphic to finite direct sums of sheaves $\cL^\sigma(\codim\sigma -i)$, where $i\leq p$. The previous proposition and corollary prove:

\begin{theorem}
  $\tau_{\leq p}$ defines a functor
\[ \tau_{\leq p}: \fM \to \fM^{\leq p}.\]
This functor is right adjoint to the inclusion functor $\fM^{\leq p} \to \fM$.
\end{theorem}

One can define similarly $\tau_{\geq p} \cF$ with a canonical surjection $\cF\to \tau_{\geq p} \cF$, and the category $\fM^{\geq p}$. The pair of categories $(\fM^{\leq 0},\fM^{\geq 0})$ satisfies all the properties of a t-structure on $\fM$, although the categories here are not triangulated.

For $\cF$ in $\fM$, define
\[ \Gr_p \cF = \tau_{\leq p}\cF / \tau_{\leq p-1} \cF, \qquad \Gr \cF = \oplus_p \Gr_p \cF.\]
Let $\fM^p$ be the full subcategory of $\fM$ consisting of sheaves isomorphic to finite direct sums of sheaves $\cL^\sigma(\codim\sigma -p)$. Then $Gr_p$ defines a functor 
\[ Gr_p: \fM\to \fM^p.\]
The category $\fM^p$ has a very simple structure:

\begin{lemma} \label{lem-Mp}
\begin{enumerate}
  \item An object $\cF$ in $\fM^p$ decomposes uniquely $\cF = \oplus_\sigma W_\sigma \otimes \cL^\sigma(\codim\sigma)$.
  \item If $\cF = \oplus_\sigma W_\sigma \otimes \cL^\sigma(\codim\sigma)$ and $\cG = \oplus_\sigma U_\sigma \otimes \cL^\sigma(\codim\sigma)$ are in $\fM^p$, then 
  \[ \Hom_{\fM^p} (\cF, \cG) = \prod_\sigma \Hom_\RR(W_\sigma,U_\sigma).\]
\end{enumerate}
\end{lemma}
\begin{proof} 
The second part of the lemma follows from the equality case in Theorem~\ref{thm-mor}. The first part of the lemma then follows from the second part applied to the identity morphism of $\cF$.
\end{proof}

The lemma implies that there is an equivalence of categories
\[ \fM^p \isom \prod_\sigma Vect_\RR.\]
In particular, $\fM^p$ is an abelian category.
Similarly, $\Gr \cF$ can be viewed as an object in the product category of graded vector spaces.

The filtration $\tau_{\leq p}$ on $\cF$ induces a filtration and grading on $\cF(\Sigma)$ and on $\overline{\cF(\Sigma)}$. In the quasi-convex case it also induces a filtration and grading on $\cF(\Sigma,\partial\Sigma)$ and on $\overline{\cF(\Sigma,\partial\Sigma)}$. As an example, when $\pi:\hat\Sigma\to\Sigma$ is a subdivision and $\cF=\pi_* \cL_{\hat\Sigma}(\codim \hat\zz)$, let 
\[ H = \overline{\cF(\Sigma)}. \]
Then $\Gr H$ has the double grading 
\[ H^i_p = \Gr_p \overline{\cF(\Sigma)}^i = \Gr_p IH^{i +\codim\zz}(\hat\Sigma).\]

\subsection{Relative Hard Lefschetz theorem.}

Let $\cF$ be in $\fM = \fM(\Sigma)$ and let $\hat{l}: \cF\to \cF(2)$ be a morphism. Then $\hat{l}$ induces a morphism on $\Gr \cF$ that we also call $\hat{l}$:
\[ \hat{l}: \Gr_p \cF \to \Gr_{p+2} \cF.\]

\begin{definition}
We say that $\hat{l}$ defines a \LO on $\Gr \cF$ if the morphism of sheaves
\[ \hat{l}^p: \Gr_{-p} \cF \to \Gr_{p} \cF \]
is an isomorphism for any $p>0$. 
\end{definition}

Let $\cF$ have a decomposition $\cF = \oplus_\sigma W_\sigma \otimes \cL^\sigma(\codim\sigma)$. Then
by Lemma~\ref{lem-Mp}, the morphism $\hat{l}: \Gr_p \cF \to \Gr_{p+2} \cF$ is given by linear maps 
\[ \hat{l}: W_\sigma^p \to W_\sigma^{p+2}.\]
Thus, $\hat{l}$ defines a \LO on $\Gr \cF$ if and only if 
\[ \hat{l}^p: W_\sigma^{-p} \to W_\sigma^{p}\]
is an isomorphism for any $p>0$ and $\sigma\in\Sigma$.

Let us specialize to the case of a subdivision $\pi:\hat\Sigma \to \Sigma$. Let $\hat{l}\in \cA^2(\hat\Sigma)$. Then $\hat{l}$ acts by multiplication on any sheaf $\cG\in\fM(\hat\Sigma)$, hence it acts on $\cF = \pi_*\cG$.

\begin{theorem}[Relative Hard Lefschetz] \label{thm-RHL1}
Let $\cG$ be in $\fM^0(\hat\Sigma)$.
Then a relatively strictly convex $\hat{l}\in\cA(\hat\Sigma)$ defines a \LO on $\Gr \pi_*\cG$.
\end{theorem}

Note that Theorem~\ref{thm-RHL} in the introduction is a special case of Theorem~\ref{thm-RHL1}, where $\hat\Sigma$ is pointed and $\cG = \cL_{\hat\Sigma}(\codim\hat\zz)$. (The grading is such that $W_\sigma^{\dim\sigma-i}$ in Theorem~\ref{thm-RHL} is equal to $W_\sigma^{-i}$ in Theorem~\ref{thm-RHL1}.) We show that the two theorems are in fact equivalent.

\begin{lemma} \label{lem-RHL-red} Theorem~\ref{thm-RHL} for all subdivisions $\pi: \hat\Sigma \to \Sigma$, where $\hat\Sigma$ is a pointed fan of dimension at most $n$, implies Theorem~\ref{thm-RHL1} for all subdivisions $\pi: \hat\Sigma \to \Sigma$ of $n$-dimensional fans and all sheaves $\cG$ in $\fM^0(\hat\Sigma)$.
\end{lemma}

\begin{proof} 
Let us first reduce to the case $\cG = \cL_{\hat\Sigma}(\codim\hat\zz)$. By the decomposition theorem on $\hat\Sigma$ it suffices to consider sheaves $\cG = \cL^\tau(\codim\tau)$ for $\tau\in \hat\Sigma$. Construct the fan $\Sigma_\tau$ consisting of cones $\sigma+\Span(\tau)$, for $\sigma\in\st\tau$, and a similar fan $\Sigma_{\pi(\tau)}$. Then $\Sigma_\tau$ is a subdivision of $\Sigma_{\pi(\tau)}$. Theorem~\ref{thm-RHL1} applied to this subdivision and the sheaf $\cG = \cL_{\Sigma_\tau}(\codim\tau)$ is equivalent to the same theorem applied to the subdivision $\pi$ and the sheaf $\cG= \cL^\tau(\codim\tau)$.

To reduce to the case where $\hat\Sigma$ is pointed, we may split the fans and write $\pi$ as 
\[ \hat\Sigma' \times \RR^m \stackrel{\pi'\times \Id}{\longrightarrow} \Sigma' \times \RR^m
,\]
where $\pi':  \hat\Sigma' \to \Sigma'$ is a subdivision and $\hat\Sigma'$ is pointed. As in Section~\ref{sec-npfans}, the operation $\otimes_A B$ maps $\fM(\Sigma')$ to $\fM(\Sigma)$ and $\fM(\hat\Sigma')$ to $\fM(\hat\Sigma)$. If Theorem~\ref{thm-RHL1} holds for the subdivision $\pi'$ and $\hat{l}'\in \cA^2(\hat\Sigma')$, then it also holds for the subdivison $\pi$ and the pullback of $\hat{l}'$. We may change $\hat{l}\in\cA^2(\hat\Sigma)$ by a global linear function so that it becomes equal to the  pullback of a relatively strictly convex function $\hat{l}'$. Then the pullback satisfies RHL if and only if $\hat{l}$ does.  
\end{proof}

\subsection{From relative to global Hard Lefschetz.} 

Consider a subdivision $\pi:\hat\Sigma\to\Sigma$ of complete fans. If $l$ is strictly convex on $\Sigma$ and $\hat{l}$ is relatively strictly convex for $\pi$, then $l+\varepsilon^2 \hat{l}$ is strictly convex on $\hat\Sigma$ for small $\varepsilon\neq 0$, and hence satisfies the Hard Lefschetz theorem on $\hat\Sigma$. We wish to prove without using convexity that if $l$ satisfies HL on $\Sigma$ (as in Theorem~\ref{thm-HL}) and $\hat{l}$ satisfies RHL, then $l+\varepsilon^2 \hat{l}$  satisfies HL on $\hat\Sigma$ for small $\varepsilon\neq 0$. 

Let us fix the notation that will be used throughout this subsection. Let $\pi:\hat\Sigma \to \Sigma$ be a subdivision, with both fans complete. Let $\cF = \pi_* \cG$ for some $\cG$ in $\fM^0(\hat\Sigma)$. Let $l\in\cA^2(\Sigma)$ and $\hat{l} \in \cA^2(\hat\Sigma)$ be such that
\begin{itemize}
  \item Multiplication with $l$ satisfies HL on $IH(\st\sigma)$ for any $\sigma\in\Sigma$ as in Theorem~\ref{thm-HL}.
  \item Multiplication with $\hat{l}$ defines a \LO on $\Gr \cF$. 
\end{itemize}

Let us decompose $\cF =  \bigoplus_\sigma W_\sigma \otimes \cL^\sigma(\codim\sigma).$ To simplify notation, define $H^* = \overline{\cF(\Sigma)}^*$ and the graded space $H^*_*$, with
\[ H^i_p = \Gr_p H^i=  \oplus_\sigma W^p_\sigma\otimes IH^{i-p +\codim\sigma}(\st\sigma).\]
Then $l$ acts on both $H^*$ and $H^*_*$ by acting on each $IH(\st\sigma)$. On the other hand, $\hat{l}$ acts on $H^*_*$ by a degree $2$ map on $W_\sigma$ for each $\sigma\in\Sigma$. This gives the map $l+\hat{l} = 1\otimes l + \hat{l}\otimes 1 : H^i_*\to H^{i+2}_*$.

\begin{lemma}
The map $l+\hat{l}$ defines a \LO on $H^*_*$:
\[ (l+\hat{l})^i: H^{-i}_* \to H^{i}_*\]
is an isomorphism for any $i>0$.
\end{lemma}

\begin{proof}
If $\hat{l}$ defines a \LO on $W_\sigma$ and $l$ defines a \LO on $IH(\st\sigma)$, then $1\otimes l + \hat{l}\otimes l$ defines a \LO on their tensor product.
\end{proof}

\begin{theorem} \label{thm-rel-global-HL}
 With notation as above, multiplication with $l+\varepsilon^2\hat{l}$ defines a \LO on $H^*$ for small $\varepsilon\neq 0$:
 \[ (l+\varepsilon^2\hat{l})^i: H^{-i} \to H^{i}\]
 is an isomorphism for any $i>0$.
\end{theorem}

\begin{proof}
The idea is to degenerate the map $l+\varepsilon^2\hat{l}$ on $H^*$ to the map $l+\hat{l}$ on $H^*_*$.

We fix a decomposition of $\cF$ as above, giving an isomorphism $\cF\isom \Gr \cF$ and $H^* \isom H^*_*$. Consider the action of $\hat{l}$ on $H^*\isom H^*_*$. If $h\in H^*_p$, then 
\[ \hat{l} (h) = g_{p+2} + g_{p+1} + g_{p}+\ldots,\]
for some $g_j\in H^*_j$.

Let us change coordinates in $H^*$ by the isomorphism $\alpha: H^*\to H^*$ that multiplies elements in $H^*_p$ with $\varepsilon^p$. In the new coordinates consider the map $\hat{l}_\varepsilon = \alpha^{-1} \circ \varepsilon^2 \hat{l} \circ \alpha$ on $H^*$. If $h\in H^*_p$ is as above, then 
\begin{align*}
 \hat{l}_\varepsilon (h) &= \alpha^{-1} \circ \varepsilon^2 \hat{l} (\varepsilon^p h) \\
 &= \alpha^{-1} \varepsilon^{p+2}(g_{p+2} + g_{p+1} + g_{p}+\ldots) \\
 &= g_{p+2} + \varepsilon g_{p+1} + \varepsilon^2 g_{p} +\ldots.
 \end{align*}
 It follows that the map $\hat{l}_\varepsilon$ on  $H^*$ has a limit as $\varepsilon\to 0$, which is the graded map $\hat{l}$ on $H^*_*$. Since $l+\hat{l}_0$ defines a \LO on $H^*_*$ by the previous lemma, so does $l+\hat{l}_\varepsilon$ on $H^*$ for small $\varepsilon\neq 0$. However, $l+\hat{l}_\varepsilon$ is isomorphic to the map $l+\varepsilon^2\hat{l}$ via the change of coordinates map $\alpha$.
\end{proof}

Consider now the action of $\hat{l}$ on $l\cdot H^*$. We replace $l\cdot H^*$ with its shift by one degree $l\cdot H^*(1)$ which is symmetric around degree $0$.

\begin{theorem}\label{thm-complete-HL}
 With notation as above, let $l\cdot H^*$ be the image of $l:H^*\to H^*$. Then multiplication with $l+\varepsilon^2\hat{l}$ defines a \LO on $l\cdot H^*(1)$ for small $\varepsilon\neq 0$:
 \[ (l+\varepsilon^2\hat{l})^{i-1}: l\cdot H^{-i} \to l\cdot H^{i-2}\]
 is an isomorphism for any $i>1$.
\end{theorem}

\begin{proof}
We first check that $l+\hat{l}$ defines a \LO on the graded space $l\cdot H^*_*(1)$. If $l$ defines a \LO on $IH^{*+\codim\sigma}(\st\sigma)$, then $l$ also defines a \LO on $l\cdot IH^{*+\codim\sigma-1}(\st\sigma)$. This implies that $l+\hat{l}$ defines a \LO on 
\[ l\cdot\big(W_\sigma\otimes IH^{*+\codim\sigma}(\st\sigma)\big)(1) = W_\sigma\otimes l\cdot IH^{*+\codim\sigma-1}(\st\sigma).\]
Now we repeat the proof of the previous theorem with $H^*$ replaced by $l\cdot H^{*}(1)$.
\end{proof}

\begin{remark} \label{rmk-gen-RHL}
Theorem~\ref{thm-RHL1} was stated for sheaves $\cG$ in $\fM^0(\hat\Sigma)$. It can be generalized to all sheaves in $\fM(\hat\Sigma)$ as follows. 

Let $\cG$ be in $\fM(\hat\Sigma)$ and let $\tilde{l}: \cG \to \cG(2)$ be a morphism that defines a Lefschetz operation on $\Gr \cG$. Let $\cF = \pi_* \cG$. Then $\cF$ can be given a triple grading
\[ \cF_{p,q}^i = \Gr_{p} \pi_* \Gr_{q} \cG^i.\]
If $F= \oplus_\sigma W_\sigma\otimes \cL^\sigma(\codim\sigma)$, then each $W_\sigma$ gets a double grading $W^{p,q}_\sigma$. The morphism $\tilde{l}$ acts on $W_\sigma$ by increasing the grading by $q$ and the multiplication with $\hat{l}\in\cA(\hat\Sigma)$ acts by increasing the grading by $p$. 

The generalization of Theorem~\ref{thm-RHL1} now states that $\tilde{l}+\hat{l}$ defines a \LO on $W_\sigma$:
\[ ( \tilde{l}+\hat{l} )^i: \bigoplus_{p+q=-i} W_\sigma^{p,q} \to \bigoplus_{p+q=i} W_\sigma^{p,q}\]
is an isomorphism for any $i>0$ and $\sigma\in\Sigma$. This claim can be reduced to Theorem~\ref{thm-RHL1} by first using a Lefschetz decomposition on $\cG$ and thus reducing the problem to the case where $W_\sigma$ is a tensor product of an $\RR[\tilde{l}]$-module and an $\RR[\hat{l}]$-module.

The morphism $\hat{l}+\varepsilon^2\tilde{l}$ acts on $\cF$, and hence also on $\Gr \cF$. We claim that it define a \LO on $\Gr\cF$ for small $\varepsilon\neq 0$. This can be proved the same way as Theorem~\ref{thm-rel-global-HL}, replacing $H$ in the proof with $W_\sigma$ and degenerating the action of $\hat{l}+\varepsilon^2\tilde{l}$ to the action of $\hat{l}+\tilde{l}$ on $W_\sigma^{*,*}$.

This generalization of the RHL theorem is useful when studying compositions of subdivisions. The resulting double filtration on $W_\sigma$ was considered in \cite{KatzStapledon}.

\end{remark}

\subsection{Intersection pairing} 

Let $\Sigma$ be a fan in $V$. 

\begin{lemma}
The duality functor $\tilde \DD: \fM\to \fM$ maps $\fM^{\leq p}$ to $\fM^{\geq -p}$.
\end{lemma}

\begin{proof}
From Lemma~\ref{lem-dualL}
\[ \tilde{\DD}(\cL^\sigma(\codim\sigma - j)) \isom \cL^\sigma (\codim\sigma + j).\]
Thus, if $j\leq p$, then $-j \geq -p$.
\end{proof}

\begin{lemma}
Assume that $\Sigma$ is quasi-convex. Let $\cF$ be in $\fM$ and consider the $A$-bilinear pairing
\[ \cF(\Sigma)\times\cF(\Sigma,\partial\Sigma) \to A,\]
given by a morphism
\[ \phi: \cF\to \tilde\DD(\cF) .\]
Then the restriction of the pairing to 
\[ \tau_{\leq p} \cF(\Sigma) \times \tau_{\leq -p-1} \cF(\Sigma,\partial\Sigma) \to A\]
is zero for every $p\in\ZZ$.
\end{lemma}

\begin{proof}
We need to show that any morphism 
\[ \tau_{\leq p} \cF  \to \tilde\DD(\tau_{\leq -p-1} \cF)\]
is zero. However, $\tilde\DD(\tau_{\leq -p-1} \cF)$ lies in $\fM^{\geq p+1}$ and there is no nonzero morphism from $\tau_{\leq p}\cF$ to it.
\end{proof}

A pairing as in the previous lemma induces a pairing on $\Gr \cF$:
\[ \Gr_p \cF(\Sigma) \times \Gr_{-p} \cF(\Sigma,\partial\Sigma) \to A.\]
If the morphism $\phi$ is an isomorphism, then the pairing on $\Gr \cF$ is nondegenerate.

Let us now return to general fans $\Sigma$, not necessarily quasi-convex. Fix volume forms  $\Omega_\sigma \in \det \sigma^\perp$ for all cones $\sigma\in\Sigma$. Recall that the volume form $\Omega_\sigma$ determines an isomorphism $\cL^\sigma\to \tilde{\DD}(\cL^\sigma)(-2\codim\sigma)$, and hence also an isomorphism
\[ \cL^\sigma(\codim\sigma) \to \tilde{\DD}(\cL^\sigma(\codim\sigma)).\]
We call this isomorphism the {\em standard isomorphism} and the resulting pairing in the quasi-convex case the 
{\em standard pairing}.

Let $\cF$ be in $\fM$ and consider a symmetric isomorphism 
\[ \phi: \cF\to \tilde\DD(\cF).\]
If $\cF = \oplus_\sigma W_\sigma \otimes \cL^\sigma(\codim\sigma)$, then $\tilde\DD(\cF) = \oplus_\sigma W^*_\sigma \otimes \tilde\DD(\cL^\sigma(\codim\sigma))$. Since $\phi$ is an isomorphism, the graded morphism $\Gr\phi$ is also an isomorphism. The graded morphism is given by the standard isomorphisms $\cL^\sigma(\codim\sigma) \to \tilde{\DD}(\cL^\sigma(\codim\sigma))$ and isomorphisms $W_\sigma\to W_\sigma^*$, that means, nondegenerate symmetric $\RR$-bilinear pairings 
\[ W_\sigma\times W_\sigma\to\RR.\]
We wish to identify the pairing on $W_\sigma$. 

\begin{lemma} \label{lem-locPD}
The isomorphism $\phi$ induces a nondegenerate symmetric bilinear pairing
\[ \overline{\cF([\sigma])}(-\codim\sigma)  \times \overline{\cF([\sigma],\partial[\sigma])}(-\codim\sigma) \to \RR.\]
This pairing restricts to the pairing on $W_\sigma$: 
\[ W_\sigma \isom \img \big( \overline{\cF([\sigma],\partial [\sigma])}(-\codim\sigma) \to \overline{\cF([\sigma])}(-\codim\sigma) \big).\]
\end{lemma}

\begin{proof}
First assume that $\Sigma=[\sigma]$ and $\codim\sigma=0$. Then $[\sigma]$ is a quasi-convex fan and the isomorphism $\phi$ defines the pairing as stated. Since $W_\sigma$ is canonically a subspace of $\overline{\cF([\sigma])}$, we claim that this pairing restricts to the pairing on $W_\sigma$. Indeed, we may reduce to the case $\cF=  W_\sigma \otimes \cL^\sigma(\codim\sigma)$, in which case the pairings are equal.

For general fans $\Sigma$ and cones $\sigma\in\Sigma$, the statements of the lemma are local, so we may replace $\Sigma$ with $[\sigma]$ and $\cF$ with $\cF|_{[\sigma]}$. The duality functor on $[\sigma]$ is different if we consider $[\sigma]$ a fan in $V$ or in $\Span\sigma$. We can use the Koszul resolution as in Lemma~\ref{lem-dualL} to compare the duality functors on $V$ and on $\Span \sigma$: 
\[ \tilde\DD_V(\cF) = \tilde\DD_{\Span\sigma} (\cF)\otimes (\det \sigma^\perp)^*.\]
We use the volume form $\Omega_\sigma$ to identify $(\det \sigma^\perp)^* \isom \RR(2\codim\sigma)$. Then
\[ \phi: \cF\to \tilde\DD_V(\cF) = \tilde\DD_{\Span\sigma} (\cF) (2\codim\sigma)\]
induces the isomorphism
\[ \cF(-\codim\sigma) \to \tilde\DD_{\Span\sigma} (\cF(-\codim\sigma)),\]
which gives the pairing on the quasi-convex $[\sigma]$ as stated. 

To see that this pairing restricts to the pairing on $W_\sigma$, note that the standard isomorphism $\cL^\sigma(\codim\sigma) \to \tilde{\DD} (\cL^\sigma(\codim\sigma))$ also differs by the factor $(\det \sigma^\perp)^*$ depending on whether we consider the fan $[\sigma]$ in $V$ or in $\Span\sigma$. This means that the map $W\to W^*$ is the same whether we construct it on $V$ or on $\Span\sigma$. The case where $[\sigma]$ is a quasi-convex fan in $\Span\sigma$ then shows that the pairing restricts to the pairing on $W_\sigma$.
\end{proof}
 
Consider the special case of a subdivision $\pi:\hat\Sigma \to \Sigma$. We fix volume forms $\Omega_\tau \in \det \tau^\perp$ for all cones $\tau\in\hat\Sigma$. Let $\cG = \cL^\tau(\codim\tau)$ in $\fM^0(\hat\Sigma)$, and $\cF = \pi_* \cG$. The standard isomorphism
\[ \phi: \cG\to \tilde\DD(\cG),\]
induces the isomorphism $\pi_*(\phi): \cF\to \tilde\DD(\cF)$. 
Let $\cF = \oplus_\tau W_\sigma \otimes \cL^\sigma(\codim\sigma)$. Then 
\[ W_\sigma \isom \img \big( \overline{\cG(\hat\sigma,\partial \hat\sigma)}(-\codim\sigma) \to \overline{\cG(\hat\sigma)}(-\codim\sigma) \big),\]
where $\hat\sigma\subset \hat\Sigma$ is the inverse image of $[\sigma]$. Now $\hat\sigma$ is a quasi-convex fan in $\Span\sigma$ and the sheaf
\[\cG(-\codim\sigma)|_{\hat\sigma} = \cL^\tau (\codim_{\Span\sigma} \tau),\]
which lies in $\fM^0([\sigma])$, has the standard pairing. 
\[ \overline{\cG(\hat\sigma)}(-\codim\sigma) \times \overline{\cG(\hat\sigma,\partial\hat\sigma)}(-\codim\sigma) \to \RR.\]
This pairing is determined by a volume form $\Omega_\tau'$, where we consider $\tau\subset\Span\sigma$. We choose the form so that $\Omega_\tau=\Omega_\tau' \cdot \Omega_\sigma$.
Lemma~\ref{lem-locPD} then gives:

\begin{lemma}
The standard pairing on $\cG(-\codim\sigma)|_{\hat\sigma}$ restricts to the pairing on $W_\sigma$.
\end{lemma}

\subsection{Relative Hodge-Riemann bilinear relations.}

Let $\pi: \hat\Sigma \to \Sigma$ be a subdivision. We fix volume forms for all cones in $\Sigma$ and $\hat\Sigma$. 

Let $\cG = \cL^\tau(\codim\tau)$ for some $\tau\in\hat\Sigma$, and let $\cF = \pi_* \cG = \oplus_\sigma W_\sigma \otimes \cL^\sigma(\codim\sigma)$. The standard isomorphism 
$\phi: \cG\to \tilde{\DD}(\cG)$ induces the isomorphism $\pi_* \phi: \cF\to \tilde{\DD}(\cF)$. The graded isomorphism $\Gr \pi_* \phi$ is given by non-degenerate symmetric bilinear pairings on $W_\sigma$. For $\hat{l} \in \cA^2(\hat\Sigma)$, define the quadratic form $Q_{\hat{l}}$ on $W^{-i}$, $i\geq 0$:
\[ Q_{\hat{l}}(w) = \langle \hat{l}^i w, w\rangle.\]

\begin{theorem}[Relative Hodge-Riemann bilinear relations] \label{thm-RHR1}

With notation as above, let $\cF = \pi_* \cL^\tau(\codim\tau)$, and let $\hat{l} \in \cA^2(\hat\Sigma)$ be relatively strictly convex with respect to $\pi$. Then the quadratic form
\[ (-1)^{\frac{c-i}{2}} Q_{\hat{l}},\]
where $c= \codim\tau-\codim\sigma$,  is positive definite on the primitive part
\[ \prim_{\hat{l}} W^{-i}_\sigma = \ker \hat{l}^{i+1}: W^{-i}_\sigma\to   W^{i+2}_\sigma\]
for any $\sigma\in\Sigma$ and $i\geq 0$.
\end{theorem}

Extending the theorem to general sheaves $\cG$ in $\fM^0(\hat\Sigma)$ is not as simple as in the case of the RHL theorem, so we will not do it. (The problem is that the sign of $Q_{\hat{l}}$ on $\prim_{\hat{l}} W^{-i}_\sigma$ depends on the dimension of $\tau$. One needs to introduce an additional filtration to take this into account.) Note that Theorem~\ref{thm-RHR} in the introduction is a special case of Theorem~\ref{thm-RHR1}, where $\hat\Sigma$ is pointed and $\cG = \cL_{\hat\Sigma}(\codim\hat\zz)$. (In that case $c=\dim\sigma$ and $W_\sigma^{\dim\sigma-i}$ in Theorem~\ref{thm-RHR} is equal to $W_\sigma^{-i}$ in Theorem~\ref{thm-RHR1}. The quadratic forms $Q_{\hat{l}}$ are the same up to the change in grading.) The two theorems are in fact equivalent. The proof of the following lemma is similar to the proof of Lemma~\ref{lem-RHL-red}, so we omit it.

\begin{lemma} \label{lem-RHR-red} Theorem~\ref{thm-RHR} for all subdivisions $\pi: \hat\Sigma \to \Sigma$, where $\hat\Sigma$ is a pointed fan of dimension at most $n$, implies Theorem~\ref{thm-RHR1} for all subdivisions $\pi: \hat\Sigma \to \Sigma$ of $n$-dimensional fans. \qed
\end{lemma}

We will also write down a generalization of Theorem~\ref{thm-convex}.

Let $\Phi$ be a convex fan with boundary $\partial\Phi$. Fix volume forms for all cones in $\Phi$.  Let $\cG = \cL^\tau(\codim\tau)$ for some $\tau\in\Phi$. Define 
\[ W_\Phi = \img \big(\overline{\cG(\Phi,\partial\Phi)} \to \overline{\cG(\Phi)}\big).\]
A piecewise linear function $\hat{l}\in \cA^2(\Phi)$ acts on $W_\Phi$ by multiplication. The  standard pairing
\[ \overline{\cG(\Phi)} \times \overline{\cG(\Phi,\partial\Phi)} \to \RR\] 
induces a nondegenerate symmetric bilinear pairing
\[ W_\Phi \times W_\Phi \to \RR.\]
Define the quadratic form $Q_{\hat{l}}$ on $W_\Phi^{-i}$, $i\geq 0$,
\[ Q_{\hat{l}}(w) = \langle \hat{l}^i w, w\rangle.\]

\begin{theorem}\label{thm-convex1}
If $\hat{l}$ is strictly convex on $\Phi$, then for any $i\geq 0$, the form 
\[ (-1)^{\frac{c-i}{2}} Q_{\hat{l}},\]
where $c=\codim\zz_\Phi - \codim\tau$, is positive definite on the primitive part
\[ \prim_{\hat{l}} W^{-i}_\Phi = \ker \hat{l}^{i+1}: W^{-i}_\Phi\to   W^{i+2}_\Phi.\]
\end{theorem}

Note that Theorem~\ref{thm-convex} in the introduction is a special case of Theorem~\ref{thm-convex1}, where $\Phi$ is pointed and $\cG = \cL_{\Phi}(\codim \zz_\Phi)$. The two theorems are in fact equivalent. The proof of this is similar to the proof of Lemma~\ref{lem-RHL-red} and we omit it. 

\begin{lemma} \label{lem-convex-red} Theorem~\ref{thm-convex} for all pointed convex fans $\Phi$ of dimension at most $n$ implies Theorem~\ref{thm-convex1} for all convex fans $\Phi$ of dimension $n$. \qed
\end{lemma}

Theorem~\ref{thm-convex1} is a local version of Theorem~\ref{thm-RHR1}. We show that the two theorems are equivalent.

\begin{lemma} \label{lem-phi-red} Theorem~\ref{thm-convex1} for all convex fans $\Phi$ of dimension at most $n$ is equivalent to Theorem~\ref{thm-RHR1} for all subdivisions of fans of dimension at most $n$.
\end{lemma}

\begin{proof} 
In the situation of Theorem~\ref{thm-RHR1}, fix a cone $\sigma\in\Sigma$ and let $\Phi\subset \hat\Sigma$ be the inverse image of $[\sigma]$. Then $W_\Phi = W_\sigma$ by Lemma~\ref{lem-prec-decomp}. The relatively strictly convex function $\hat{l}$ on $\hat\Sigma$ restricts to a strictly convex function on $\Phi$. The pairing and the quadratic forms $Q_{\hat{l}}$ agree on $W_\Phi$ and $W_\sigma$ by Lemma~\ref{lem-locPD}. Therefore Theorem~\ref{thm-convex1} implies Theorem~\ref{thm-RHR1}. The converse is also true by considering $\Phi$ as a subdivision of the fan $[\sigma]$, where $\sigma=|\Phi|$.
\end{proof}

Since HR relations imply the HL theorem, it follows that Theorem~\ref{thm-convex}(2) implies Theorems 1.1-1.3 in the introduction, as well as their generalizations, Theorems~\ref{thm-RHL1}, \ref{thm-RHR1}, \ref{thm-convex1}.

\subsection{From relative to global Hodge-Riemann bilinear relations.}

Let $\pi:\hat\Sigma \to \Sigma$ be a subdivision, with both fans complete. Let $\cF=\pi_*\cL^\tau(\codim\tau)$ for some $\tau\in\hat\Sigma$. Choose a decomposition $\cF = \oplus_\sigma W_\sigma \otimes \cL^\sigma(\codim\sigma)$. 
Let $l\in\cA^2(\Sigma)$ and $\hat{l} \in \cA^2(\hat\Sigma)$ be such that
\begin{itemize}
  \item The quadratic form $Q_l$ satisfies HR on $IH(\st\sigma)$ for any $\sigma\in\Sigma$ as in Theorem~\ref{thm-HR}.
  \item The quadratic form $Q_{\hat{l}}$ satisfies RHR on $\cF$  as in Theorem~\ref{thm-RHR1}. 
\end{itemize}

As before, define $H^i = \overline{\cF(\Sigma)}^i$ and the graded space $H^*_*$, with
\[ H^i_p = \Gr_p H^i=  \oplus_\sigma W^p_\sigma\otimes IH^{i-p +\codim\sigma}(\st\sigma).\]
We have a non-degenerate pairing on $H^*_*$ and a linear map $l+\hat{l}: H^i_*\to H^{i+2}_*$. Using these, we define the quadratic form $Q_{l+\hat{l}}$ on $H^{-i}_*$ for $i\geq 0$. 

\begin{lemma}
The quadratic form 
\[ (-1)^{\frac{c-i}{2}} Q_{l+\hat{l}},\]
where $c=\codim\tau$, is positive definite on 
\[ \prim_{l+ \hat{l}} H^{-i}_* = \ker (l+\hat{l})^{i+1}: H^{-i}_*\to   H^{i+2}_*\]
for any $i\geq 0$.
\end{lemma}

\begin{proof}
The direct sum decomposition $H^*_*=\oplus_\sigma W_\sigma\otimes IH(\st\sigma)$ is orthogonal with respect to the pairing and the quadratic form $Q_{l+\hat{l}}$. On a summand $W_\sigma\otimes IH(\st\sigma)$ the degree $2$ map and the pairing come from degree $2$ maps and parings on each factor. Since by assumption both $Q_{\hat{l}}$ and $Q_l$ satisfy HR, so does $Q_{l+\hat{l}}$.
\end{proof}

Multiplication with $l+\varepsilon^2 \hat{l}$ defines a degree $2$ map on $H^*$. We also have a nondegenerate symmetric bilinear pairing on $H^*$, coming from the standard pairing on $\overline{\cG(\hat\Sigma)}$. Let us use this map and pairing to define the quadratic form $Q_{l+\varepsilon^2 \hat{l}}$ on $H^{-i}$ for $i\geq 0$.

\begin{theorem} \label{thm-rel-global-HR}
 The quadratic form 
\[ (-1)^{\frac{c-i}{2}} Q_{l+\varepsilon^2\hat{l}},\]
where $c=\codim\tau$, is positive definite on 
\[ \prim_{l+ \varepsilon^2 \hat{l}} H^{-i} = \ker (l+\varepsilon^2\hat{l})^{i+1}: H^{-i}\to   H^{i+2}\]
for any $i\geq 0$ and small $\varepsilon\neq 0$. 
\end{theorem}

\begin{proof}
As in the proof of Theorem~\ref{thm-rel-global-HL} we consider the change of coordinates isomorphism $\alpha: H^*\to H^*$ that multiplies $H^*_p$ with $\varepsilon^p$. Then the map $L_\varepsilon = \alpha^{-1} \circ (l+ \varepsilon^2\hat{l}) \circ \alpha$ defines a \LO on $H^*$ for $\varepsilon$ small, with $L_0$ being the graded map $l+\hat{l}$ on $H^*_*$.
Similarly, we consider the pairing on $H^*$,  $\langle h,g\rangle = \langle \alpha(h),\alpha(g)\rangle$. This pairing limits to the graded pairing on $H^*_*$ because a pairing between $h\in H^*_{-p}$ and $g\in H^*_{p+i}$ for $i>0$ contains a factor $\varepsilon^i$ and hence goes to zero as $\varepsilon \to 0$. Thus, via the change of coordinates map we have a quadratic form $Q_\varepsilon$ on $H^*$ that limits to the form $Q_{l+\hat{l}}$ on $H^*_*$ as $\varepsilon \to 0$. Since the form $Q_{l+\hat{l}}$ on $H^*_*$ satisfies HR, it is nondegenerate and has the correct signature in each degree. This implies that for small $\varepsilon\neq 0$, $Q_\varepsilon$ also satisfies HR. However,  $Q_\varepsilon = Q_{l+\varepsilon^2\hat{l}}$ via the change of coordinates map $\alpha$.
\end{proof}

Let $l\cdot H^*$ be the image of $l:H^*\to H^*$. Define the quadratic form $Q_{l+\varepsilon^2\hat{l}}$ on $l\cdot H(1)$ by 
\[ Q_{l+\varepsilon^2\hat{l}}(l h) = \langle (l+\varepsilon^2\hat{l})^{i-1} lh, h\rangle, \qquad h\in H^{-i}, i> 0.\]
This form is well-defined (does not depend on how we lift $lh$ to $h$) because the pairing is $\cA(\Sigma)$-bilinear.

\begin{theorem}\label{thm-complete-HR}
 The quadratic form 
\[ (-1)^{\frac{c-i}{2}} Q_{l+\varepsilon^2\hat{l}},\]
 where $c=\codim\tau$, is positive definite on 
\[ \prim_{l+ \varepsilon^2 \hat{l}} l\cdot H^{-i} = \ker (l+\varepsilon^2\hat{l})^{i}: l\cdot H^{-i}\to   l\cdot H^{i}\]
for any $i\geq 0$ and small $\varepsilon\neq 0$. 
\end{theorem}

\begin{proof}
We first check that $Q_{l+\hat{l}}$ satisfies HR on the graded space $l\cdot H^*_*(1)$. This reduces to the fact that if $Q_l$ satisfies HR on a cohomology space $IH$, then $Q_l$ also satisfies HR on $(l\cdot IH)(1)$.
Now we repeat the proof of the previous theorem with $H^*$ replaced by $l\cdot H^{*}(1)$.
\end{proof}

Theorems~\ref{thm-complete-HL} and \ref{thm-complete-HR} can be compared to Theorem~\ref{thm-complete} in the case where $\hat\Sigma$ is pointed and $\cF= \pi_*\cL_{\hat\Sigma}(\codim\hat\zz)$. Theorem~\ref{thm-complete} states that any strictly convex $\hat{l}$ satisfies HL and HR on $l\cdot IH(\hat\Sigma)$. Theorems~\ref{thm-complete-HL} and \ref{thm-complete-HR} state that this is true for strictly convex functions that are close to $l$, namely functions of the form $l+\varepsilon^2 \hat{l}$, where $\hat{l}$ is relatively strictly convex. With the argument that we used above we can not yet prove Theorem~\ref{thm-complete} in general.

\begin{remark}
The RHL and RHR theorems become especially simple in the case of a semi-small subdivision. Consider a subdivision $\pi:\hat\Sigma\to\Sigma$, where both fans are pointed and $\hat\Sigma$ is simplicial. Then $\pi$ is called semi-small if $\dim\pi(\sigma) \leq 2\dim\sigma$ for any $\sigma\in\hat\Sigma$; it is called small if strict inequality holds for any nonzero $\sigma\in\hat\Sigma$.

If $\pi:\hat\Sigma\to\Sigma$ is semi-small then $\pi_* \cL_{\hat\Sigma} (\codim\hat\zz)$ lies in $\fM^0(\Sigma)$. This can be seen as follows. Let $\pi_* \cL_{\hat\Sigma} (\codim\hat\zz) = \oplus_\tau W_\tau\otimes \cL^\tau(\codim\tau)$, where 
\[ W_\tau = \img \big(IH(\hat\tau,\partial\hat\tau)\to IH(\hat\tau)\big).\]
The space $IH(\hat\tau,\partial\hat\tau)$ is generated by elements $\chi_\sigma$, where $\pi(\sigma) = \tau$. Since $\deg\chi_\sigma = 2 \dim\sigma \geq \dim\tau$, this implies that $W_\tau$ is zero in negative degrees. By \Po duality it is also zero in positive degrees, hence $W_\tau=W_\tau^0$. A similar argument shows that if $\pi$ is small then $\pi_* \cL_{\hat\Sigma} = \cL_{\Sigma}$.

Since $\pi_* \cL_{\hat\Sigma} (\codim\hat\zz)$ lies in $\fM^0(\Sigma)$, the RHL theorem trivially holds for it, hence a strictly convex $l$ on $\Sigma$ defines a \LO on $IH(\hat\Sigma)$. The RHR relations for $\pi$ imply that the intersection pairing on $W_\tau$ is definite and then $Q_l$ satisfies Hodge-Riemann bilinear relations on $IH(\hat\Sigma)$. However, the fact that the intersection pairing on $W_\tau$ is definite requires the existence of a strictly relatively convex $\hat{l}$, even though this $\hat{l}$ does not enter into the statement. The proof of the RHR theorem below also requires the existence of such $\hat{l}$.
\end{remark}

\section{Proofs of RHL and RHR theorems.}

In this section we prove Theorems~\ref{thm-RHL}-\ref{thm-complete} in the introduction using induction on dimension. The first three of these theorems were generalized to non-pointed fans and more general sheaves in Theorems \ref{thm-RHL1}, \ref{thm-RHR1}, \ref{thm-convex1}, but these generalizations were shown to be equivalent to the theorems in the introduction.

Since HR implies HL, we get from Lemma~\ref{lem-phi-red} that part (2) of Theorem~\ref{thm-convex} implies Theorems~\ref{thm-RHL}-\ref{thm-convex}. Thus, it remains to prove Theorem~\ref{thm-convex}(2) and Theorem~\ref{thm-complete}. We do this by the following steps:
\begin{itemize}
  \item[Step 1.]  Show that Theorem~\ref{thm-convex} in dimension $n$ follows from Theorem~\ref{thm-complete} in dimension $n-1$.
  \item[Step 2.] Reduce Theorem~\ref{thm-complete} to the case where $\hat\Sigma$ is simplicial.
  \item[Step 3.] Prove Theorem~\ref{thm-complete} for simplicial fans $\hat\Sigma$. 
\end{itemize}

\subsection{Reduction to Theorem~\ref{thm-complete} in dimension $n-1$.}

Let $\Phi$ be a pointed convex fan as in Theorem~\ref{thm-convex}. We complete $\Phi$ to $\Phi^c$ by choosing a ray $\rho$ such that $-\rho \in \Int|\Phi|$ and adding to $\Phi$ cones $\rho + \sigma$, where $\sigma\in\partial\Phi$. We may assume that $\hat{l}$ is positive on $|\Phi|\setmin \{0\}$. Then the extension of $\hat{l}$ to $\Phi^c$ so that $\hat{l}$ is zero on $\rho$ is strictly convex on $\Phi^c$.

The idea of the reduction is the following. Given $h\in IH^{n-i}(\Phi,\partial\Phi)$ such that its image in $W_\Phi$ is primitive, we can consider $h$ as an element in $IH^{n-i} (\Phi^c)$ 
and compare $Q_{\hat{l}} (h)$ in $W_\Phi$ with $Q_{\hat{l}}(h)$ in $IH(\Phi^c)$. 
The problem is that $h$ may not be primitive in $IH^{n-i} (\Phi^c)$. 
However, $h+\chi_\rho g$ is primitive for a suitable $g\in IH^{n-i-2}([\st\rho]) = IH^{n-i-2}(\hat\Sigma)$. Thus, when computing $Q_{\hat{l}} (h)$, we get an extra term involving $g$. 

Note that $\Phi^c$ has a local product structure at $\rho$ (recall the definition of local product structure in Section~\ref{sec-products}). Let $\hat\Sigma$ be the image of $\st\rho$ by the projection from $\Span\rho$. Since $\hat{l}$ vanishes on $\rho$, it is the pullback of a strictly convex function on $\hat\Sigma$. As in Theorem~\ref{lem-loc-mult}, we choose $\chi_\rho$ such that the \Po pairings on $\hat\Sigma$ and $[\st\rho]$ satisfy
\[ \langle\cdot,\cdot\rangle_{\hat\Sigma} = \langle\chi_\rho \cdot,\cdot\rangle_{[\st\rho]}.\]
Let $l\in \cA^2(\hat\Sigma)$ be such that its pullback to $[\st\rho]$ agrees with $\chi_\rho$ up to a global linear function. Since $\Phi$ is convex, it follows that $l$ is also convex, with graph equal to $\partial\Phi$.

\begin{lemma}
With notation as in Theorems~\ref{thm-convex}-\ref{thm-complete}, let $h\in IH^{n-i}(\Phi,\partial\Phi) \subset IH^{n-i} (\Phi^c)$ be such that its image in $IH^{n-i}(\Phi)$ lies in $\prim_{\hat{l}} W_\Phi$. Then there exists $g\in IH^{n-i-2}(\hat\Sigma)$ such that
\begin{enumerate}
  \item $l\cdot g \in \prim_{\hat{l}} l\cdot IH^{n-i-2}(\hat\Sigma).$
  \item $h+\chi_\rho g \in \prim_{\hat{l}}  IH^{n-i}(\Phi^c).$
  \item $Q_{\hat{l}}(h+\chi_\rho g) = Q_{\hat{l}}(h)  +  Q_{\hat{l}}(l\cdot g)$.\\
  Here the the three quadratic forms $Q_{\hat{l}}$ are defined on $IH(\Phi^c)$, $W_\Phi$ and $l\cdot IH(\hat\Sigma)$, respectively.
\end{enumerate}
\end{lemma}

\begin{proof}
Since $\hat{l}^{i+1} h$ vanishes in $IH(\Phi)$, it follows that when we consider $h$ as an element of $IH(\Phi^c)$ then 
\[ \hat{l}^{i+1} h \in IH^{n+i+2}([\st\rho],\partial [\st\rho]) = \chi_\rho IH^{n+i}(\hat\Sigma).\]
By HL applied to $IH(\hat\Sigma)$, the map 
\[ \hat{l}^{i+1}: IH^{n-i-2}(\hat\Sigma) \to IH^{n+i}(\hat\Sigma) \]
is an isomorphism. Hence we can find $g$ that satisfies (2):
\[  \hat{l}^{i+1} (h+\chi_\rho g) = 0  \text{ in $IH(\Phi^c)$}. \]

Note that $\hat{l}^{i+1} (\chi_\rho g) = -\hat{l}^{i+1} (h)$ is supported on $\Phi$, hence $\hat{l}^{i+1} (\chi_\rho g) = 0$ in $IH([\st\rho]) = IH(\hat\Sigma)$. This implies (1).  

Part (3) follows from the fact that $\hat{l}^{i+1} (\chi_\rho g)$ is supported on $[\st\rho]$ and $\hat{l}^{i+1}(h)$ is supported on $\Phi$. Hence $\chi_\rho g$ and $h$ are orthogonal with respect to $Q_{\hat{l}}(h)$ on $IH(\Phi^c)$.
\end{proof}

\begin{lemma} If Theorem~\ref{thm-complete} holds for $\hat\Sigma$, then Theorem~\ref{thm-convex}(2) holds for $\Phi$.
\end{lemma}

\begin{proof}
With notation as in the previous lemma, we need to show that 
\[ (-1)^{\frac{n-i}{2}} Q_{\hat{l}}(h) \geq 0 \]
and equality holds only if $h$ is zero in $IH(\Phi)$. From part (3) of the previous lemma we get
\[ (-1)^{\frac{n-i}{2}}Q_{\hat{l}}(h) = (-1)^{\frac{n-i}{2}}Q_{\hat{l}}(h+\chi_\rho g) -  (-1)^{\frac{n-i}{2}}Q_{\hat{l}}(l\cdot g). \]
The first term on the right hand side is non-negative by HR applied to $\Phi^c$. Moreover, it is zero only if $h=-\chi_\rho g$, which implies that $h$ is zero in $IH(\Phi)$. The second term on the right hand side is non-negative by applying part (2) of Theorem~\ref{thm-complete} to $l\cdot g$. Indeed,
\[ -  (-1)^{\frac{n-i}{2}}Q_{\hat{l}}(l\cdot g) = (-1)^{\frac{n-i-2}{2}}Q_{\hat{l}}(l\cdot g) \geq 0.\]
\end{proof}

Notice that if $\Phi$ has dimension $n$, then $\hat\Sigma$ has dimension $n-1$, hence by induction we may assume that $\hat\Sigma$ satisfies Theorem~\ref{thm-complete}. It remains to prove Theorem~\ref{thm-complete} in dimension $n$, assuming the other theorems in dimension $n$ and less. More precisely, let $\Sigma$ be the fan on which $l$ is strictly convex. Then $\hat\Sigma$ is a subdivision of $\Sigma$, and we may assume that Theorems~\ref{thm-RHL} and \ref{thm-RHR} hold for this subdivision. 

Part (2) of Theorem~\ref{thm-complete} implies part (1) of the same theorem. We show that the two parts are equivalent. (The proof of this equivalence is the only place where we use Theorems~\ref{thm-RHL} and \ref{thm-RHR} in the proof of Theorem~\ref{thm-complete}.)

\begin{lemma}\label{lem-equiv2}
In Theorem~\ref{thm-complete} part (1) implies part (2). 
\end{lemma}

\begin{proof}
We fix the function $l$, but vary $\hat{l}$ in the connected set of all strictly convex functions on $\hat\Sigma$. By part (1) of the theorem, the quadratic form $Q_{\hat{l}}$ is nondegenerate on $l\cdot IH^i(\hat\Sigma)$ and hence has the same signature for all $\hat{l}$. This means that to prove part (2) of the theorem for all $\hat{l}$, it suffices to prove it for one strictly convex $\hat{l}$. By Theorem~\ref{thm-complete-HR}, we know part (2) of the theorem for strictly convex functions of the form $l+\varepsilon^2\tilde{l}$, where $\tilde{l}$ is relatively strictly convex and $\varepsilon\neq 0$ is small.
\end{proof}

\subsection{Proof of Theorem~\ref{thm-complete} for simplicial fans.}

We start with a more general situation where the fan may be nonsimplicial. 

\begin{lemma} \label{lem-mcmullen}
With notation as in Theorem~\ref{thm-complete}, assume that there is a cone $\tau\in\hat\Sigma$ such that $\hat\Sigma$ has a local product structure at every ray $\rho\in\hat\Sigma\setmin[\tau]$. Suppose some $l h$ violates part (1) of the theorem:
\[ lh \in \ker l^{i-1}: l\cdot IH^{n-i}(\hat\Sigma) \to  l\cdot IH^{n+i-2}(\hat\Sigma).\]
Then $lh$ is zero in $IH([\st\rho])$ for any $\rho\in \hat\Sigma\setmin[\tau]$.
\end{lemma}

\begin{proof}
We change $\hat{l}$ by a linear function so that it vanishes on $\tau$ and write it as
\[ \hat{l} = \sum_{\rho\in \hat\Sigma\setmin[\tau]} a_\rho \chi_\rho\]
for some $a_\rho>0$. Here $\chi_\rho$ are chosen as in Lemma~\ref{lem-loc-mult}.

Consider the restriction map 
\[ IH(\hat\Sigma) \to IH([\st\rho]) = IH(\overline{\st\rho}),\]
and the isomorphism 
\[ IH([\st\rho],\partial [\st\rho]) \isom \chi_\rho \cdot IH([\st\rho]).\]
These maps are compatible with the action of $\hat{l}$. 

Since $lh \in ker(\hat{l}^{i-1})$, we get 
\[ l h|_{l\cdot IH(\overline{\st\rho})} \in \prim_{\hat{l}} l\cdot IH(\overline{\st\rho}).\]
Now
\[ 0 = Q_{\hat{l}}(l\cdot h) = \langle \sum a_\rho \chi_\rho \hat{l}^{i-2} lh, h\rangle_{\hat\Sigma}  = \sum a_\rho Q_{\hat{l}}(lh|_{l\cdot IH(\overline{\st\rho})}).\]
Applying Theorem~\ref{thm-complete}(2) to $\overline{\st\rho}$ by induction on dimension, all terms in the sum have the same sign, hence they must be zero, implying that $lh|_{l\cdot IH(\overline{\st\rho})} =0$ for every $\rho$.
\end{proof}

\begin{lemma} 
Theorem~\ref{thm-complete} holds for simplicial fans $\hat\Sigma$.
\end{lemma}

\begin{proof}
A simplicial fan has a local product structure at every ray $\rho$. Hence we may take $\tau=\hat\zz$ in the previous lemma and get that if $lh$ violates part (1) of the theorem then it is zero in $IH([\st\rho])$ for any ray $\rho\in\hat\Sigma$. We claim that this implies $lh$ is zero in $IH(\hat\Sigma)$. Indeed, $\cL(\hat\Sigma) = \cA(\hat\Sigma)$ is generated by $\chi_\rho$ as an $\RR$-algebra. Since the multiplication of $lh$ and $\chi_\rho$ is zero, this implies that the \Po pairing between $lh$ and any other element is zero.

This proves part (1) of Theorem~\ref{thm-complete}. Part (2) follows by Lemma~\ref{lem-equiv2}.
\end{proof}

\subsection{Reduction of Theorem~\ref{thm-complete} to the simplicial case.}

We consider a sequence of star subdivisions of $\hat\Sigma$ that gives the barycentric subdivision $b( \hat\Sigma)$. We claim that Theorem~\ref{thm-complete} for $b( \hat\Sigma)$ implies the same theorem for $\hat\Sigma$. 

With a small abuse of notation, let us assume that $\tilde{\Sigma}\to \hat\Sigma$ is one star subdivision at a cone $\tau\in\hat\Sigma$. Let the new ray be $\rho\in\tilde{\Sigma}$. We assume that $\hat\Sigma$ has a local product structure at $\tau$ and the fan $\overline{\st\tau} = \st\tau/\Span\tau$ is simplicial. This assumption holds for each star subdivison in the barycentric subdivision. We let $l$ on $\tilde{\Sigma}$ be pulled back from $\hat\Sigma$, and we define the strictly convex function $\tilde{l} = \hat{l} - \varepsilon^2 \chi_\rho$ on $\tilde{\Sigma}$, where $\varepsilon\neq 0$ is small. We claim that if $(\tilde{\Sigma},l,\tilde{l})$ satisfies Theorem~\ref{thm-complete}, then so does $(\hat{\Sigma},l,\hat{l})$.

Define a new fan $\overline{\Sigma}$ as follows. Take $[\st \tau] \subset \hat\Sigma$ and complete it by adding cones $\bar\rho + \sigma$, where $\bar\rho = -\rho$ and $\sigma \in \partial [\st \tau]$. 

Notice that $\st\rho \subset \tilde{\Sigma}$ and $\st\bar\rho\subset \overline{\Sigma}$ are isomorphic as ringed spaces over $A$, hence we may identify intersection cohomology sheaves on these spaces. Let $\hat{h} \in \cL(\hat\Sigma)$ and $\tilde{h}\in \cL(\tilde\Sigma)$ be such that  $\hat{h}=\tilde{h}$ on $\hat\Sigma\setmin\st\tau = \tilde\Sigma \setmin \st\rho$. Then we construct $\bar{h} = \beta(\hat{h},\tilde{h})\in \cL(\overline{\Sigma})$, letting $\bar{h}$ be $\hat{h}$ on $[\st\tau]$ and $\tilde{h}$  on $[\st\bar\rho]\isom [\st\rho]$. 

We apply the same procedure to construct $\bar{l} = \beta(\hat{l},\tilde{l})$ and $l = \beta(l,l)$ in $\cA^2(\overline{\Sigma})$. The resulting $l$ is convex and $\bar{l}$ is strictly convex. It is proved in \cite{Karu,BL2} that 
\[ Q_{\hat{l}} (\hat{h}) = Q_{\tilde{l}} (\tilde{h}) + Q_{\bar{l}} (\bar{h}).\]

\begin{lemma} 
Let $l \hat{h} \in l\cdot \cL(\hat\Sigma)$ be such that its class is primitive:
\[ l [h] \in \prim_{\hat{l}} l\cdot IH^{n-i}(\hat\Sigma).\]
Then there exists $l \tilde{h} \in l\cdot \cL(\tilde\Sigma)$, equal to $l \hat{h}$ on $\tilde\Sigma \setmin \st\rho$, and $l \bar{h} = \beta(l\hat{h}, l\tilde{h}) \in l\cdot \cL(\overline\Sigma)$ such that
\begin{enumerate}
  \item $l [\tilde{h}] \in \prim_{\hat{l}} l\cdot IH^{n-i}(\tilde\Sigma).$
  \item $l [\bar{h}] \in \prim_{\hat{l}} l\cdot IH^{n-i}(\overline\Sigma).$
\end{enumerate}
Moreover, if $l [\tilde{h}] = 0$ then $l [\hat{h}] = 0$.
\end{lemma}

\begin{proof}
We embed $\cL_{\hat\Sigma} \hookrightarrow \cL_{\tilde\Sigma}$. Then considering $l[\hat{h}] \in IH(\tilde\Sigma)$, we have
\[ \tilde{l}^i l [\hat{h}]  = \tilde{l}^i l [\hat{h}] - \hat{l}^i l [\hat{h}] = (\tilde{l}^i - \hat{l}^i) l [\hat{h}] \in l\cdot IH^{n+i+2} ([\st\rho],\partial[\st\rho]) = \chi_\rho l\cdot IH^{n+i} (\overline{\st\rho}).\]
Here $\overline{\st\rho}$ is the projection of $\st\rho$ from the span of $\rho$.
Applying Theorem~\ref{thm-complete}(1) to $\overline{\st\rho}$ by induction, we know that 
\[ \tilde{l}^i: l\cdot IH^{n-i} (\overline{\st\rho})\to l\cdot IH^{n+i} (\overline{\st\rho})\]
is an isomorphism, hence we can find $lh\in l\cdot \cL^{n-i}(\overline{\st\rho})$ such that 
\[ \tilde{l}^i (l[\hat{h}] +\chi_\rho l[h]) = 0.\]
This gives $\tilde{h} = \hat{h}+\chi_\rho h$ such that $l[\tilde{h}]$ is primitive.

Clearly $\beta(l\hat{h}, l\tilde{h}) = l \beta(\hat{h},\tilde{h})$ has the form $l\bar{h}$. Let us check that its class is primitive. Recall that we considered $l\hat{h}\in l\cdot \cL(\tilde{\Sigma})$. Now $\beta(l\hat{h},l\hat{h})$ lies in $l\cdot \cL(\overline{\Sigma})$ and the construction of $l\bar{h}$ from it is the same as the construction of $l\tilde{h}$ from $l\hat{h}$, hence $l\bar{h}$ is primitive by the same reason that $l\tilde{h}$ is primitive.

Let us prove the last statement. Suppose $l[\tilde{h}] =0$. Then 
\[ l[\hat{h}]\in IH^{n-i+2}([\st\tau],\partial[\st\tau]).\]
 Assuming that $\dim\tau\geq 2$,  $\hat{l}^i$ is  injective on $IH^{n-i+2}([\st\tau],\partial[\st\tau])$ by Lemma~\ref{lem-loc-prod}.
\end{proof}

\begin{lemma}
If Theorem~\ref{thm-complete} holds for $(\tilde{\Sigma},l,\tilde{l})$, then it also holds for $(\hat{\Sigma},l,\hat{l})$.
\end{lemma}

\begin{proof}
With notation as in the previous lemma, we get 
\[ Q_{\hat{l}} (l\hat{h}) = Q_{\tilde{l}} (l\tilde{h}) + Q_{\bar{l}} (l \bar{h}).\]
If both terms on the right hand side have the correct signs then so does the left hand side. The first term on the right hand side has the correct sign by the assumption that 
Theorem~\ref{thm-complete} holds for $(\tilde{\Sigma},l,\tilde{l})$. 
For the second term we need to show that Theorem~\ref{thm-complete} also holds for $(\overline{\Sigma},l,\bar{l})$. It suffices to prove part (1) of Theorem~\ref{thm-complete} for $(\overline{\Sigma},l,\bar{l})$. If 
\[ l \bar{g} \in \ker \bar{l}^{i-1}: l\cdot IH^{n-i} (\overline{\Sigma}) \to l\cdot IH^{n+i-2} (\overline{\Sigma}),\]
then by Lemma~\ref{lem-mcmullen}, $l\bar{g}=0$ in $IH(\st\bar\rho)$, hence $l\bar{g}\in IH^{n-i+2}([\st\tau],\partial[\st\tau])$. By Lemma~\ref{lem-loc-prod}, the map
\[ \bar{l}^{i-1}: IH^{n-i+2} ([\st\tau],\partial[\st\tau]) \to IH^{n+i}([\st\tau],\partial[\st\tau])\]
is injective, hence $l\bar{g}=0$.
\end{proof}

\bibliographystyle{plain}
\bibliography{relHL}

\begin{thebibliography}{10}

\bibitem{BBFK}
Gottfried Barthel, Jean-Paul Brasselet, Karl-Heinz Fieseler, and Ludger Kaup.
\newblock Combinatorial intersection cohomology for fans.
\newblock {\em Tohoku Math. J. (2)}, 54(1):1--41, 2002.

\bibitem{BBFK2}
Gottfried Barthel, Jean-Paul Brasselet, Karl-Heinz Fieseler, and Ludger Kaup.
\newblock Combinatorial duality and intersection product: a direct approach.
\newblock {\em Tohoku Math. J. (2)}, 57(2):273--292, 2005.

\bibitem{BL1}
Paul Bressler and Valery~A. Lunts.
\newblock Intersection cohomology on nonrational polytopes.
\newblock {\em Compositio Math.}, 135(3):245--278, 2003.

\bibitem{BL2}
Paul Bressler and Valery~A. Lunts.
\newblock Hard {L}efschetz theorem and {H}odge-{R}iemann relations for
  intersection cohomology of nonrational polytopes.
\newblock {\em Indiana Univ. Math. J.}, 54(1):263--307, 2005.

\bibitem{Brion}
Michel Brion.
\newblock The structure of the polytope algebra.
\newblock {\em Tohoku Math. J. (2)}, 49(1):1--32, 1997.

\bibitem{dCMM}
Mark~Andrea de~Cataldo, Luca Migliorini, and Mircea Musta{\c{t}}{\v{a}}.
\newblock The combinatorics and topology of proper toric maps.
\newblock {\em arXiv 1407.3497}.

\bibitem{Fulton}
William Fulton.
\newblock {\em Introduction to toric varieties}, volume 131 of {\em Annals of
  Mathematics Studies}.
\newblock Princeton University Press, Princeton, NJ, 1993.
\newblock The William H. Roever Lectures in Geometry.

\bibitem{Karu}
Kalle Karu.
\newblock Hard {L}efschetz theorem for nonrational polytopes.
\newblock {\em Invent. Math.}, 157(2):419--447, 2004.

\bibitem{KatzStapledon}
Eric Katz and Alan Stapledon.
\newblock Local {$h$}-polynomials, invariants of subdivisions, and mixed
  {E}hrhart theory.
\newblock {\em Adv. Math.}, 286:181--239, 2016.

\bibitem{Oda}
Tadao Oda.
\newblock {\em Convex bodies and algebraic geometry}, volume~15 of {\em
  Ergebnisse der Mathematik und ihrer Grenzgebiete (3) [Results in Mathematics
  and Related Areas (3)]}.
\newblock Springer-Verlag, Berlin, 1988.
\newblock An introduction to the theory of toric varieties, Translated from the
  Japanese.

\bibitem{Stanley}
Richard~P. Stanley.
\newblock Subdivisions and local {$h$}-vectors.
\newblock {\em J. Amer. Math. Soc.}, 5(4):805--851, 1992.

\end{thebibliography}

\end{document}